\documentclass[11pt]{article}
\usepackage[T1]{fontenc}
\usepackage{geometry}
\usepackage{color}
\usepackage{float}
\usepackage{amsmath}
\usepackage{amssymb}
\usepackage{amsthm}
\usepackage{mathrsfs}
\usepackage{setspace}
\usepackage[all,cmtip]{xy}
\usepackage{tikz}
\usepackage{tikz-cd}
\usepackage{enumerate}
\usepackage{mathtools}
\usepackage{lipsum}
\usepackage{url}

\newtagform{roman}[]()

\newcommand{\Z}{{\mathbb{Z}}}
\newcommand{\Q}{{\mathbb{Q}}}
\newcommand{\R}{{\mathbb{R}}}
\newcommand{\C}{{\mathbb{C}}}

\newcommand{\fm}{{\mathfrak m}}
\newcommand{\fn}{{\mathfrak n}}
\newcommand{\fr}{{\mathfrak r}}
\newcommand{\ft}{{\mathfrak t}}
\newcommand{\fd}{{\mathfrak d}}
\newcommand{\fp}{{\mathfrak p}}
\newcommand{\fP}{{\mathfrak P}}
\newcommand{\fq}{{\mathfrak q}}
\newcommand{\fu}{{\mathfrak u}}
\newcommand{\fv}{{\mathfrak v}}

\newcommand{\cO}{{\mathscr O}}
\newcommand{\cM}{{\mathscr M}}
\newcommand{\cF}{{\mathscr F}}
\newcommand{\cA}{{\mathscr A}}
\newcommand{\cC}{{\mathscr C}}
\newcommand{\cD}{{\mathscr D}}
\newcommand{\cL}{{\mathscr L}}
\newcommand{\cT}{{\mathcal T}}
\newcommand{\cS}{{\mathcal S}}
\newcommand{\cP}{{\mathscr P}}
\newcommand{\ep}{{\varepsilon}}
\newcommand{\GL}{{\mathrm {GL}}}

\newcommand{\Id}{{\mathrm {Id}}}

\newcommand{\Cl}{{\mathrm {Cl}}}
\newcommand{\rN}{{\mathrm {N}}}

\newtheorem{thm}{Theorem}[section]
\newtheorem{prop}[thm]{Proposition}
\newtheorem{conj}[thm]{Conjecture}

\newtheorem{lem}[thm]{Lemma}
\newtheorem{rem}[thm]{Remark}

\theoremstyle{definition}

\newtheorem{lemma}{Lemma}

\numberwithin{equation}{section}

\begin{document}

\large{\centerline{\textbf{Log-free zero density estimates for automorphic $L$-functions}}}

\vspace{5mm}

\large{\centerline{Chen An}}
\vspace{3mm}
\normalsize{}
\begin{abstract}
We prove a log-free zero density estimate for automorphic $L$-functions defined over a number field $k$. This work generalizes and sharpens the method of pseudo-characters and the large sieve used earlier by Kowalski and Michel. As applications, we demonstrate for a particular family of number fields of degree $n$ over $k$ (for any $n$) that an effective Chebotarev density theorem and a bound on $\ell$-torsion in class groups hold for almost all fields in the family.
\end{abstract}

\begin{section}{Introduction}\label{intro}

Many problems in number theory require an estimate for the number of zeros of an $L$-function inside the critical strip, or in a region near the line $\Re(s)=1$. Log-free zero density estimates for Dirichlet $L$-functions near the line $\Re(s)=1$ were established by Gallagher \cite{Gal70} using Turan's power sum method, and by Selberg \cite{Sel72} using the concept of a pseudo-character (see also Jutila \cite{Jut78}). For automorphic $L$-functions, Kowalski and Michel \cite{KM02} gave a log-free zero density estimate for families of such $L$-functions associated to a set of automorphic representations over $\Q$. Improved estimates in this context have recently been given by Lemke Oliver and Thorner \cite{LT15}, Brumley, Thorner, and Zaman \cite{TZ18}, and Thorner and Zaman \cite{TZ19}. In this paper, we generalize and sharpen the method of Kowalski and Michel, applied to automorphic $L$-functions over a number field $k$.

Let $k$ be a number field and let an integer $n$ be fixed. For each integer $q \ge 1$, let $S(q)$ be a set of irreducible cuspidal automorphic representations of $\GL_n(\mathbb{A}_k)$. Let $A>0$ be such that $\mathrm{Cond}(f) \le q^A$ for every $f \in S(q)$ and for all $q \ge 1$. (See \cite{IS00} or (\ref{condf}) for the definition of $\mathrm{Cond}(f)$, the analytic conductor associated to $f$.) There exists $d>0$ such that $|S(q)| \ll q^d$ for all $q \ge 1$; we can in fact choose $d=2nA+\ep$ for any $\ep>0$; see Remark \ref{autrepbd}.
For $\alpha <1$, $T \ge 0$, we define the region of interest $M(\alpha,T)$ as
\begin{equation}\label{mat}
M(\alpha,T)=\{z \in \C: \alpha \le \Re(z) \le 1, \ |\Im(z)| \le T\}.
\end{equation}
For any cuspidal automorphic representation $f$ of $\GL_n(\mathbb{A}_k)$ with associated automorphic $L$-function $L(f,s)$, we define the zero counting function $N(f;\alpha,T)$ as 
\begin{equation*}
N(f;\alpha,T)=| \{ \rho \in M(\alpha,T): L(f,\rho)=0 \} |.
\end{equation*}
Our main theorem is a log-free zero density estimate for automorphic $L$-functions associated to automorphic representations in $(S(q))_{q \ge 1}$ over $k$.

\begin{thm}\label{main}
Fix a number field $k$ with $n_k=[k:\Q]$ and a set $(S(q))_{q \ge 1}$ of irreducible cuspidal automorphic representations of $\GL_n(\mathbb{A}_k)$ with associated data $A,d$ as above. Assume that the Ramanujan-Petersson conjecture and a uniform bound on Langlands parameters at Archimedean places hold for any $f \in (S(q))_{q \ge 1}$. Let $\alpha \ge \frac34$ and $T \ge 2$. Then for all $q \ge 1$ and for any $\ep>0$,
\begin{equation}\label{mainbd}
\sum_{f \in S(q)} N(f;\alpha,T) \ll_{\ep} (q^{c_1+\ep}T^{c_2+\ep})^{1-\alpha},
\end{equation}
where we may take
\begin{equation*}
c_1=2d+4nA+\frac{A}{2}+1, \ \ c_2=\frac{nn_k}{2}+3.
\end{equation*}
\end{thm}

Let us compare Theorem \ref{main} to several recent results. First, we mention the preprint \cite[Theorem 1.2]{TZ19}.
Theorem \ref{main} is weaker than \cite[Theorem 1.2]{TZ19} because we assume the Ramanujan-Petersson conjecture and a uniform bound on Langlands parameters at Archimedean places. However, note that the exponents in our bound (\ref{mainbd}) are smaller than the exponents in the bound $(q^AT^{n_k})^{10^7n^4(1-\alpha)}$ obtained in \cite{TZ19}. Our method is also an application of a large sieve, but of a different flavor from \cite{TZ19}.

Let us also compare Theorem \ref{main} to Theorem 2 of Kowalski and Michel \cite{KM02}, which motivates our method. Our result improves the analogous result of \cite[Theorem 2]{KM02} in terms of the $T$-dependence on $\alpha$: we obtain a bound $(q^{c_1+\ep}T^{c_2+\ep})^{1-\alpha}
$, as compared to their bound $T^Bq^{c_0\frac{1-\alpha}{2\alpha-1}}$ for some $B>0$. This improvement occurs in our version of the large sieve (Theorem \ref{14}), in which we refine the estimation of certain dyadic sums. Furthermore, our bound strengthens the work of \cite{KM02} by generalizing the result to any base field $k$. We hope the presentation of this paper will clarify the use of pseudo-characters in the setting of automorphic $L$-functions over $k$. Finally, we also mention a thesis of Lai \cite{Lai19}, which worked on adapting the method of Kowalski and Michel.

Theorem \ref{main} may be used to prove an effective Chebotarev density theorem for families of number fields, as shown first in \cite{PTW17}. We demonstrate a particular case of such a deduction.

For a number field $K/k$, we denote by $\widetilde{K}$ the Galois closure of $K$ over $k$ within a fixed choice of $\overline{\Q}$. For any field $K$, we denote $D_K=|\mathrm{Disc}(K)|$. 
For a number field $k$ and a Galois extension $L$ of $k$, we define the prime ideal counting functions
$\pi(x):=\left| \{ \fp \subseteq \cO_k: \mathrm{Nm}_{k/\Q}(\fp) \le x\} \right|,$ and
\begin{equation}\label{picount}
\pi_\cC(x,L/k):=\left| \{ \fp \subseteq \cO_k: \fp \text{ is unramified in }L,\left[ \frac{L/k}{\fp} \right]=\cC, \mathrm{Nm}_{k/\Q}(\fp) \le x\} \right|.
\end{equation}
Here $\displaystyle{\left[ \frac{L/k}{\fp} \right]}$ is the Artin symbol and $\cC$ is any fixed conjugacy class in $\mathrm{Gal}(L/k)$.

We prove the following effective Chebotarev density theorem for a particular family of degree $n$ extensions of $k$, for any $n \ge 2$.
\begin{thm}\label{cdt}
Fix $k$ and $n \ge 2$. Let $Z_n^\ast(k;X)$ denote the set of totally ramified cyclic Galois extensions $K/k$ with $\mathrm{Gal}(K/k) \cong C_n$ (the cyclic group of order $n$), and $\mathrm{Nm}_{k/\Q}\mathrm{Disc}(K/k) \le X$.
For every $\ep>0$, aside from at most $\ll_{\ep} X^{\ep}$ possible exceptions, each field $K \in Z_n^\ast(k;X)$ has the property that for every conjugacy class $\cC \subseteq C_n$,
\begin{equation*}
\left| \pi_\cC(x,\widetilde{K}/k)-\frac{|\cC|}{n}\pi(x) \right| \ll
\begin{cases}
\frac{|\cC|}{n}x^{1-\kappa} & \text{if } (\log D_{\widetilde{K}})^{2/\kappa} \le x<D_{\widetilde{K}}^{1/(24\kappa)}, \\
\frac{|\cC|}{n} \frac{x}{\exp(c_3(\log x)^{1/2}n^{-1/2})} & \text{if } x \ge D_{\widetilde{K}}^{1/(24\kappa)},
\end{cases}
\end{equation*}
where $c_3>0$ is an absolute constant, $\kappa>0$ is a constant depending only on $n,n_k$ and on $\ep$. 
Moreover, for some $c_n>0$, $|Z_n^\ast(k;X)| \sim c_n X^{\frac{1}{n-1}}$ as $X \to \infty$.
\end{thm}
Note that if $n$ is a prime, then every extension $K/k$ with $\mathrm{Gal}(K/k) \cong C_n$ is totally ramified.

Brumley, Thorner, and Zaman \cite[Theorem 2.4]{TZ18}, and Thorner-Zaman \cite[Theorem 2.1]{TZ19b} have proved significant new results of this flavor. The theorem \cite[Theorem 2.4]{TZ18} strengthens \cite{PTW17} and \cite[Theorem 2.1]{TZ19b} generalizes \cite{PTW17} in settings over $\Q$. We only remark that in the special case of cyclic extensions we consider here, Theorem \ref{cdt} is more general as it allows an arbitrary base field $k$.

Theorems analogous to Theorem \ref{cdt} may be obtained from Theorem \ref{main} for other families of fields, but these are conditional on certain other results; see Section \ref{sectcdt}. Recent work such as \cite{TZ19b} removes some significant restrictions by other means and thus we do not elaborate extensively on the general setting.

While Theorem \ref{cdt} follows the philosophy of \cite{PTW17} (incorporating a refinement from \cite{TZ18}), a key lemma used in that work over $\Q$ fails to be true over $k$ in general, and thus we must introduce new ideas to accommodate this; see Section \ref{genub}.

As a second application of Theorem \ref{main}, we bound $\ell$-torsion in class groups of fields in the family $Z_n^\ast(k;X)$. Given a number field $K$, the ideal class group $\Cl_K$ is the quotient group of the fractional ideals modulo principal ideals. For an integer $\ell \ge 1$, we define the $\ell$-torsion subgroup (written multiplicatively)
\begin{equation*}
\Cl_K[\ell]=\{ [\mathfrak{a} ] \in \Cl_K: [\mathfrak{a} ]^\ell=\Id \}.
\end{equation*}
We prove the following bound on $\ell$-torsion.
\begin{thm}\label{ltorsion}
Fix a number field $k/\Q$ and an integer $n \ge 2$.
Let $Z_n^\ast(k;X)$ be as in Theorem \ref{cdt}. For every sufficiently small $\ep>0$, for every $X \ge 1$,
aside from at most $\ll_{n,\ep} X^{\ep}$ possible exceptions, each field $K \in Z_n^\ast(k;X)$ has the property that for every integer $\ell \ge 1$,
\begin{equation}\label{ltorbd}
|\Cl_K[\ell]| \ll_{n,n_k,D_k,\ell,\ep} D_K^{\frac12-\frac{1}{2\ell(n-1)}+\ep}.
\end{equation}
\end{thm}
Note that Theorem \ref{ltorsion} improves a result of Frei and Widmer \cite{FW17} since the possible exceptional set allowed by Theorem \ref{ltorsion} is smaller; in \cite{FW17} the bound (\ref{ltorbd}) is obtained but with possibly $ \ll X^{\frac{1}{n-1}-\min\{ \frac{1}{2\ell(n-1)},\tilde{\delta}\}+\ep}$ exceptional cases in $Z_n^\ast(k;X)$ (for some $\tilde{\delta}=\tilde{\delta}(n,n_k)>0$). Theorem \ref{ltorsion} also strengthens Theorem 2.1(iv) in \cite{TZ19b} since we can take any base field. 
Theorem \ref{ltorsion} coincides with \cite[Theorem 2.4(1)]{LOTZ} in the case when the order of the cyclic group is prime. In \cite[Theorem 2.4(2)]{LOTZ}, the authors provide the $\ell$-torsion bounds when $K/k$ is of degree $n$ and the Galois group is $S_n$. The machinery of \cite{LOTZ} implies our Theorems \ref{cdt} and \ref{ltorsion}.
For other recent progress on $\ell$-torsion bounds, see e.g., \cite{PTW19} and \cite{Wan21}.

The outline of the paper is as follows. In Section 2, we describe the key components of the proof of Theorem \ref{main}, including a large sieve inequality (Theorem \ref{14}). In Section 3, we show how to deduce Theorem \ref{14} from Theorem \ref{13}, a dyadic large sieve inequality. Note that this is the step in which we make our key improvement to the method of Kowalski-Michel. In Section 4, we prove the technical details of Theorem \ref{13}. In Section 5, we deduce Theorem \ref{main} from the large sieve inequality (Theorem \ref{14}). In Section 6, we prove Theorem \ref{cdt} and in Section 7 we briefly deduce Theorem \ref{ltorsion} from Theorem \ref{cdt}.

We use Vinogradov's notation $f \ll_{\nu} g$, which means $|f| \le c(\nu)|g|$, where $c(\nu)>0$ is a constant that may depend on parameter $\nu$. The notation $f \asymp_{\nu} g$ means that $f \gg_\nu g$ and $f \ll_\nu g$.

\begin{rem}\label{autrepbd}
The expected value for $d$ in Theorem \ref{main} is $A(n+1)$; this is known for $n<3$ by \cite[Theorems 1.1 and 1.2]{BM18}. The best known value of $d$ for general $n$ is $2nA+\ep$ for any $\ep>0$ by \cite[Theorem A.1]{TZ18}. With this known value, the constant $c_1$ in the bound (\ref{mainbd}) is $8nA+\frac{A}{2}+1$.
\end{rem}

\end{section}

\begin{section}{Method of proof of the main theorem}

\begin{subsection}{The general setting of zero density estimates}
To situate the method of proof of our main theorem (Theorem \ref{main}) in the world of zero density estimates for families of $L$-functions, and tools to prove them, we very briefly recall a few highlights.

In a classical setting, that of Dirichlet $L$-functions associated to Dirichlet characters $\chi$ mod $q$, one can prove zero density estimates using an argument analogous to \cite[Sections 9 and 10.2]{IK04}. Formally, one can consider zero-counting functions such as
\begin{equation*}
N_q(\alpha,T)=\sum_{\chi \text{ mod } q} N(\alpha,T,\chi), \ \ N(\alpha,T,\chi)=| \{ \rho \in \C: L(\rho,\chi)=0, \Re(\rho) \ge \alpha, |\Im(\rho)| \le T \} |.
\end{equation*}
In this case, the Huxley density estimate is
\begin{equation*}
N_q(\alpha,T) \ll (qT)^{\frac{12}{5}(1-\alpha)}(\log qT)^A
\end{equation*}
where $A$ is an absolute constant; see \cite[Section 18.2]{IK04}.
A ``hybrid'' density estimate of 
Montgomery \cite{Mon69} is
\begin{equation*}
\sum_{q \le Q} \ \sideset{}{_{}^{\ast}}\sum_{\chi (\mathrm{mod} \ q)} N(\alpha,T,\chi) \ll (Q^2T)^{\frac52 (1-\alpha)}(\log QT)^c,
\end{equation*}
for some $c>0$, where $^\ast$ restricts the sum to primitive characters.

In some applications of the zero density estimates, one needs the estimate to be log-free. For example, in the proof of the Linnik's theorem (see e.g., Chapter 18 of \cite{IK04}), one needs a log-free zero density estimate on average for Dirichlet $L$-functions of the form
\begin{equation*}
N_q(\alpha,T) \ll (qT)^{c(\alpha)(1-\alpha)}.
\end{equation*}
As another example, in \cite{Mor73}, a log-free zero density estimate is a vital part in the so-called Hoheisel property; see \cite{HT20}.

There can also be log-free zero density estimates on average, which are hybrid in conductor, such as
\begin{equation}\label{jut}
\sum_{q \le Q} \ \sideset{}{_{}^{\ast}}\sum_{\chi (\mathrm{mod} \ q)} N(\alpha,T,\chi) \ll (Q^{c_1}T^{c_2})^{1-\alpha}
\end{equation}
for some $c_1,c_2>0$.
Gallagher \cite{Gal70} proved this hybrid version (\ref{jut}) for some constants $c_1,c_2$ and Selberg \cite{Sel72} proved $c_1=5+\ep$, $c_2=3+\ep$ suffice.
Jutila \cite{Jut78} 
refined the estimate and obtained $c_1=4+\ep$, $c_2=2+\ep$; these works are closely related to the method of this paper.

Our focus is on automorphic $L$-functions. A log-free zero density estimate on average for automorphic $L$-functions can take, for example, the form
\begin{equation*}
\sum_{f \in S(q)} N(f;\alpha,T) \ll T^Bq^{c(\alpha)(1-\alpha)}
\end{equation*}
for some constants $B,c(\alpha)$. As we mentioned, such an inequality was first proved in \cite{KM02}. This form was then improved in works such as \cite{LT15} \cite{TZ18} \cite{TZ19}, in various settings, to have right-hand side of the form $(q^{c_1}T^{c_2})^{(1-\alpha)}$ for some constants $c_1,c_2$. This stronger form is also the outcome of our method.

Two main approaches to prove a log-free zero density estimate for $L$-functions are as follows. One approach uses Turan's power sum inequality; see e.g. \cite{Gal70}, \cite{LT15}, and \cite{TZ19}. Another approach uses pseudo-characters; see e.g. \cite{Sel72}, \cite{Jut78}, and \cite{KM02}. 

All the pseudo-characters in the above papers have an almost orthogonality property (as a generalization of the orthogonality property for Dirichlet characters) and lead to large sieve inequalities. In fact, a key idea of proving zero density estimates is to give an upper bound for the absolute value of a certain Dirichlet polynomial, and large sieve inequalities are well-suited to this purpose.
Our Theorem \ref{14} below is the large sieve inequality we obtain from the almost orthogonality property of pseudo-characters, which in our work are denoted $\psi_{f,\fr}(\fn)$, defined in (\ref{psidef}). This definition of pseudo-characters is analogous to that in \cite{KM02}, and the almost orthogonality can be seen in Lemma \ref{12}. 

\end{subsection}

\begin{subsection}{Preliminaries}\label{notations1}
To describe the key components of the proof of Theorem \ref{main}, we need to define certain $L$-functions related to automorphic representations.
In this section we briefly summarize necessary details on the irreducible cuspidal automorphic representations $f \in S(q)$ and the associated $L$-functions. For reference, see, e.g, \cite{Bru06}, \cite{TZ19}.

For each $f \in S(q)$, the associated $L$-function is 
\begin{equation}\label{raun}
L(f,s)=\prod_{\fp} \prod_{j=1}^{n}(1-\alpha_j(\fp)\rN\fp^{-s})^{-1}=L^{\mathrm{ur}}(f,s)L^{\mathrm{ra}}(f,s).
\end{equation}
Here $\alpha_j(\fp)$ are the Satake parameters of $f$; the unramified $L$-function attached to $f$ is
\begin{equation*}
L^{\mathrm{ur}}(f,s)=\prod_{\fp \nmid \fq_f} \prod_{j=1}^{n}(1-\alpha_j(\fp)\rN\fp^{-s})^{-1}=\sum_\fm \frac{\lambda_f(\fm)}{\rN\fm^s},\end{equation*}
(see (\ref{arcondf}) for the arithmetic conductor $\fq_f$) where 
the coefficients in the sum satisfy $\lambda_f(\fm)=0$ if $(\fm,\fq_f) \subsetneq (1)$; the ramified $L$-function attached to $f$ is
\begin{equation*}
L^{\mathrm{ra}}(f,s)=\prod_{\fp|\fq_f} \prod_{j=1}^{n}(1-\alpha_j(\fp)\rN\fp^{-s})^{-1}.
\end{equation*}
Since $f$ is assumed to satisfy the Ramanujan-Petersson conjecture in the setting of Theorem \ref{main}, $|\alpha_j(\fp)| \le 1$ for all $j,\fp$, and $|\lambda_f(\fp)| \le n$ for any unramified $\fp$. Specifically, as a finite product, $L^{\mathrm{ra}}(f,s)$ has no zeros in the region $\Re(s)>0$. Thus, the zeros of $L(f,s)$ in the region $M(\alpha,T)$ defined in (\ref{mat}) are exactly those of $L^{\mathrm{ur}}(f,s)$ and when we construct $Z(f)$ as a collection of $\eta$-well-spaced zeros, it is a collection of $\eta$-well-spaced zeros of $L^{\mathrm{ur}}(f,s)$.

For $f$ as above and any $g \in S(q)$ with
\begin{equation*}
L(g,s)=\prod_{\fp} \prod_{j=1}^{n}(1-\beta_j(\fp) \rN\fp^{-s})^{-1}=L^{\mathrm{ur}}(g,s)L^{\mathrm{ra}}(g,s),
\end{equation*}
we have the Rankin-Selberg L-function
\begin{equation*}
L(f \times g,s)=\prod_{\fp} \prod_{i=1}^{n} \prod_{j=1}^{n}(1-\alpha_i(\fp)\beta_j(\fp)\rN\fp^{-s})^{-1}=L^{\mathrm{ur}}(f \times g,s)L^{\mathrm{ra}}(f \times g,s)
\end{equation*}
where the unramified $L$-function attached to $f \times g$ is
\begin{equation*}
L^{\mathrm{ur}}(f \times g,s)=\prod_{\fp \nmid \fq_{f \times g}} \prod_{i=1}^{n} \prod_{j=1}^{n}(1-\alpha_i(\fp)\beta_j(\fp)\rN\fp^{-s})^{-1}=\sum_\fm \frac{\lambda_{f \times g}(\fm)}{\rN\fm^s},
\end{equation*}
(see (\ref{arcondfg}) for the arithmetic conductor $\fq_{f \times g}$) and 
the coefficients in the sum satisfy $\lambda_{f \times g}(\fm)=0$ if $(\fm,\fq_{f \times g}) \subsetneq (1)$. Correspondingly, the ramified $L$-function associated to $f \times g$ is
\begin{equation*}
L^{\mathrm{ra}}(f \times g,s)=\prod_{\fp|\fq_{f \times g}} \prod_{i=1}^{n} \prod_{j=1}^{n}(1-\alpha_i(\fp)\beta_j(\fp) \rN\fp^{-s})^{-1}.
\end{equation*}
Since $L(f,s)$ and $L(g,s)$ are assumed to satisfy the Ramanujan-Petersson conjecture, $|\alpha_i(\fp)| \le 1$, $|\beta_j(\fp)| \le 1$ for all $i,j,p$. Therefore, $L^{\mathrm{ra}}(f \times g,s)$ is analytic and has no zeros in $\Re(s)>0$.

For $f \in S(q)$, at each Archimedean place $v$ of $k$, the local $L$-factor of $f$ is defined to be
\begin{equation*}
L(f_v,s) = \prod_{i=1}^n \Gamma_{k_v}  (s+\mu_f(v, i))
\end{equation*}
where $\{ \mu_f(v, i) \}_{i=1}^n$ are Langlands parameters associated to $k,v$ and $\Gamma_\R(s)=\pi^{-s/2}\Gamma(s/2)$, $\Gamma_\C(s)=2(2\pi)^{-s}\Gamma(s).$ 
Similarly, for $g \in S(q)$, we can define
\begin{equation*}
L(g_v,s) = \prod_{j=1}^n \Gamma_{k_v}  (s+\mu_g(v, j)).
\end{equation*}
Then at each Archimedean place $v$ of $k$, the local $L$-factor of $f \times g$ is defined to be
\begin{equation*}
L(f_v \times g_v,s) = \prod_{i=1}^n \prod_{j=1}^n \Gamma_{k_v}  (s+\mu_{f \times g}(v, i, j)).
\end{equation*}
The parameters $\{ \mu_{f \times g}(v, i, j) \}$ is equal to the set of parameters $\mu_f(v,i)\mu_g(v,j)$ and are uniformly bounded by our assumption.

 We also denote the completed $L$-functions
 \begin{equation}\label{arcondf}
  \Lambda(f,s)=L(f, s) (D_k^{n}\rN\fq_{f})^{s/2}\prod_{v \in S_\infty} L(f_v,s), 
  \end{equation}
\begin{equation}\label{arcondfg}
 \Lambda(f \times g,s)=L(f \times g, s) (D_k^{n^2}\rN\fq_{f \times g})^{s/2}\prod_{v \in S_\infty} L(f_v \times g_v,s). 
 \end{equation}
where $\fq_f,\fq_{f \times g}$ are the arithmetic conductors of $f$, $f \times g$, respectively. Note that the analytic conductor is defined to be
\begin{equation}\label{condf}
 \mathrm{Cond}(f, t)=D_k^{n} \rN\fq_{f} \prod_{v \in S_\infty } \prod_{i=1}^n  (1+|it+\mu_{f}(v,i)|^{d(v)}), \ \ \mathrm{Cond}(f)=\mathrm{Cond}(f, 0),
 \end{equation}
\begin{equation}\label{condfg}
 \mathrm{Cond}(f \times g, t)=D_k^{n^2} \rN\fq_{f \times g} \prod_{v \in S_\infty } \prod_{i=1}^n \prod_{j=1}^n  (1+|it+\mu_{f \times g}(v,i,j)|^{d(v)}), \ \ \mathrm{Cond}(f \times g)=\mathrm{Cond}(f \times g, 0) 
 \end{equation}
where $d(v)=1$ if $k_v=\R$ and $d(v)=2$ if $k_v=\C$.
 Then there are the functional equations (see e.g., \cite{Bru06})
\begin{equation}\label{fedifform}
\Lambda(f,s)=\ep(f)\Lambda(\overline{f},1-s), \ \ \Lambda(f \times g,s)=\ep(f \times g)\Lambda(\overline{f \times g},1-s)
\end{equation}
where $\ep(f), \ep(f \times g)$ are the root numbers, complex numbers of modulus 1.

\end{subsection}

\begin{subsection}{A large sieve inequality: Theorem \ref{14} and its application}\label{secofls}
The key step to prove Theorem \ref{main} is a large sieve inequality (Theorem \ref{14}). 
We briefly define objects that appear in our large sieve inequality, and then we state it and give an overview of how it implies Theorem \ref{main}.

For two ideals $\fu,\fv$ of $\cO_k$, $(\fu, \fv)$ means the smallest ideal in $\cO_k$ containing $\fu$ and $\fv$, and $\fu | \fv$ means $\fv \subset \fu$. 

Let $(S(q))_{q \ge 1}$ be as in Theorem \ref{main} with associated data $A,d$, and fix $q \ge 1$.
To each $f \in S(q)$, we associate an unramified $L$-function 
\[ L^{\mathrm{ur}}(f,s)=\sum_{(\fn,\fq_f)=(1)} \frac{\lambda_f(\fn)}{\rN\fn^s} \]
 where $\fq_f$ is the arithmetic conductor of $f$ (see e.g., \cite{TZ19} or (\ref{arcondf})). There is also an associated unramified Rankin-Selberg $L$-function $L^{\mathrm{ur}}(f \times \overline{f},s)$, and we define
\begin{equation}\label{sf}
s(f):=\mathrm{Res}_{s=1} L^{\mathrm{ur}}(f \times \overline{f},s).
\end{equation}

\begin{rem}\label{rm1}
From \cite{MW89},  we know that $L(f \times g,s)$ extends to a meromorphic function on $\C$. It has no poles unless $g=\overline{f}$, in which case its only pole is simple and is at $s=1$. Hence, $s=1$ is a simple pole of $L^{\mathrm{ur}}(f \times \overline{f},s)$. Therefore, the residue $s(f)$ defined in (\ref{sf}) exists and is nonzero.
\end{rem}

We fix $z \ge 1$ (a parameter depending only on $n,n_k$, to be chosen later in (\ref{zvalue})) and let $\fP=\displaystyle{\prod_{\rN\fp<z} \fp}$, where the product runs over all prime ideals $\fp$ of $\cO_k$ such that $\rN\fp < z$.
We fix $0<\delta<\frac14$ (a parameter to be chosen later in (\ref{par1})) and define
\begin{equation}\label{rfdef}
R(f)=\{\fr  \ | \ \fr \text{ is squarefree}, (\fr, \fq_f \fP)=1, \text{and } |\lambda_f(\fp)|>\rN\fp^{-\delta} \text{ for each } \fp|\fr\}.
\end{equation} 
For any ideal $\fn$ in $\cO_k$, define the M\"{o}bius function for the field $k$ as
\begin{equation*}
\mu_k(\fn)=\begin{cases} (-1)^m & \text{ if } \fn=\fp_1\dots \fp_m \text{ with } \fp_j \text{ distinct,} \\ 0 & \text{ if } \fn \text{ is not squarefree.}  \end{cases}
\end{equation*}
For $\fr$ such that $\lambda_f(\fr) \neq 0$, we define
\begin{equation}\label{psidef}
\psi_f(\fr):=\mu_k(\fr)\rN\fr|\lambda_f(\fr)|^{-2},\ \ \psi_{f,\fr}(\fn):=\mu_k(\fn)^2\psi_f((\fn,\fr)).
\end{equation}
Note that the arithmetic functions $\lambda_f$, $\psi_f$, and $\psi_{f,\fr}$ are multiplicative. The function $\psi_{f,\fr}$ plays the role of pseudo-characters; see Lemma \ref{12}.

Let $\alpha, T$ be fixed as in Theorem \ref{main}, and consider the region $M(\alpha,T)$ in the critical strip.
We say that elements in a fixed set $Z$ of complex numbers are $\eta$-well-spaced if for any two numbers $\rho \neq \rho^{'} \in Z$,
\[ |\Im(\rho-\rho^{'})| \ge \eta. \]
For each $f \in S(q)$, suppose $Z(f)$ is a set of zeros of $f$ in $M(\alpha,T)$ that are $\eta$-well-spaced with 
\begin{equation}\label{realeta}
\eta=\frac{C}{\log qT},
\end{equation}
where $C$ is an absolute constant defined in (\ref{cvalue}). In Section \ref{redu}, we will show that to prove Theorem \ref{main}, it suffices to restrict our attention to such well-spaced sets of zeros. In particular, when $Z(f)$ is appropriately chosen for each $f \in S(q)$,
\begin{equation}\label{redubd}
\sum_{f \in S(q)} N(f;\alpha,T) \ll (qT)^{1-\alpha} \sum_{f \in S(q)} |Z(f)|.
\end{equation}
With (\ref{redubd}), the proof of Theorem \ref{main} reduces to proving
\begin{equation}\label{main2}
\sum_{f \in S(q)} |Z(f)| \ll_{\ep} (q^{c_1-1+\ep}T^{c_2-1+\ep})^{1-\alpha}.
\end{equation}

We now (slightly informally) state the three key components required to prove Theorem \ref{main}, reserving the definitions of some parameters to later. We will use a Dirichlet polynomial 
\begin{equation}\label{zr}
z_\fr(f,s):=\sum_{\substack{\fn \\ w \le \rN\fn \le x \\ (\fn,\fP)=(1) \\ \fn \text{ squarefree}}} a_\fn \psi_{f,\fr}(\fn)\lambda_f(\fn) \rN\fn^{-s}
\end{equation}
as our zero detector, for carefully chosen coefficients $a_\fn$ (see (\ref{par0})) and parameters $w,x$ depending on $q,T$ (see (\ref{par1}) and (\ref{par2})). In particular, when $\rho$ is a zero of $L(f,s)$ in the region $M(\alpha,T)$, we will show $|z_\fr(f,\rho)|$ is bounded away from zero (see Proposition \ref{18}). This allows us to prove (see Remark \ref{rema}):
\begin{lemma}\label{181} 
Fix $q \ge 1$ and let $S(q)$ be as in Theorem \ref{main}. Suppose for each $f \in S(q)$, that $Z(f)$ is an $\eta$-well-spaced set of zeros with $\eta$ as in (\ref{realeta}). For our choices of $a_\fn,R,w,x$ (see (\ref{par0}), (\ref{par1}), and (\ref{par2})),
\begin{equation*}
\sum_{f \in S(q)} |Z(f)| \ll \frac{1}{\log R} \sum_{f \in S(q)} \frac{1}{s(f)} \sum_{\rho \in Z(f)} \sum_{\substack{\fr \in R(f) \\ \rN\fr \le R}} \frac{1}{|\psi_f(\fr)|} \left| z_\fr(f,\rho) \right|^2.
\end{equation*}
\end{lemma}

On the other hand, we will use a large sieve inequality to show that as $\rho$ ranges over zeros of $L(f,s)$ and $f$ runs over our family $S(q)$, $|z_\fr(f,\rho)|$ cannot be away from zero too often, leading to our log-free zero density estimate. 
The large sieve inequality is as follows.

\begin{thm}\label{14}
Fix $S(q)$ and $Z(f)$ as in Lemma \ref{181}. Let $\delta$ be as used in the definition in (\ref{rfdef}). Assume $T \ge 2, R \ge 2$ and $N$ has the property that
there exists $0<\ep_0<\frac14$ such that
\begin{equation}\label{3.10}
N>M:=2\left( q^{d+nA/2}TR^{1+3\delta}(\log R) \right)^{\frac{1}{\frac{1}{2}-\ep_0}}.
\end{equation}
Then for any complex numbers $a_\fn$ such that $a_\fn=0$ if $\rN\fn<M$,
\begin{eqnarray}\label{14eqn}
\nonumber & & \sum_{f \in S(q)} \frac{1}{s(f)} \sum_{\rho \in Z(f)} \sum_{\substack{\fr \in R(f) \\ \rN\fr \le R}} \frac{1}{|\psi_f(\fr)|} \left| \sum_{\rN\fn \le N} a_\fn \psi_{f,\fr}(\fn)\lambda_f(\fn)\rN\fn^{-\rho} \right|^2 \\
& \ll & \log(qTN)\left( 1+\log \frac{\log N}{\log qTR} \right) \sum_{\rN\fn \le N} |a_\fn|^2\rN\fn^{1-2\alpha}.
 \end{eqnarray}
\end{thm}

We remark that the motivation for the choice of weights $\frac{1}{s(f)}$ and $\frac{1}{|\psi_f(\fr)|}$ is visible in (\ref{wtmot}). We also remark that for the application of Theorem \ref{14}, we will set the support of coefficients $a_\fn$ only on squarefree ideals $\fn$.

The final piece required to obtain Theorem \ref{main} is then (see Remark \ref{remb}):
\begin{lemma}\label{201}
For our choices of $a_\fn,R,w,x$ (see (\ref{par0}), (\ref{par1}), and (\ref{par2})), we have
\begin{equation}\label{201eqn}
\log(qTx)\left( 1+\log \frac{\log x}{\log qTR} \right) \sum_{w \le \rN\fn \le x} |a_\fn|^2 \rN\fn^{1-2\alpha} \ll (\log qT)x^{2(1-\alpha)}.
\end{equation}
\end{lemma}
We will finally choose $w=M, x=N$. Combining (\ref{redubd}), Lemma \ref{181}, Theorem \ref{14}, and Lemma \ref{201}, we essentially obtain Theorem \ref{main}; see Section \ref{pfofmainthm} for details, and precise definitions of parameters. 

In Theorem \ref{main}, an important observation is that we obtain a much better bound in the $T$-aspect in the log-free zero density estimate, compared with Theorem 2 of \cite{KM02}. This is done by carefully treating the large sieve inequality in Theorem \ref{14}. In particular, in the proof of Theorem \ref{14}, we still use a dyadic sum but shorten the sum so that our estimate becomes finer; see Theorem \ref{13} for details.

\end{subsection}

\end{section}

\begin{section}{Deduction of Theorem \ref{14} from a dyadic large sieve inequality}\label{lgsv}
The key to proving the large sieve (Theorem \ref{14}) is the following dyadic version. We use the same notation established in Section \ref{secofls}.

\begin{thm}\label{13}
Fix $q \ge 1$ and let $S(q)$ be as in Theorem \ref{main}. Suppose for each $f \in S(q)$, that $Z(f)$ is an $\eta$-well-spaced set of zeros with $\eta$ as in (\ref{realeta}). Assume that $1<\tau \le 2$, $R \ge 2$, $0<\ep_0<\frac14$, $0<\delta<\frac14$ (as in (\ref{rfdef})), and that $N'$ is such that
\begin{equation}\label{3.100}
N'>M':=\left( q^{d+nA/2}(\tau-1)^{-1}R^{1+3\delta}(\log R) \right)^{\frac{1}{\frac{1}{2}-\ep_0}}.
\end{equation}
Then for any sequence of $a_\fn \in \C$,
\begin{equation}\label{13eqn}
 \sum_{f \in S(q)} \frac{1}{s(f)} \sum_{\substack{\fr \in R(f) \\ \rN\fr \le R}} \frac{1}{|\psi_f(\fr)|} \left| \sum_{N' \le \rN\fn \le \tau N'} a_\fn \psi_{f,\fr}(\fn)\lambda_f(\fn) \right|^2 \ll (\tau-1)N'\sum_{N' \le \rN\fn \le \tau N'}|a_\fn|^2.
\end{equation}
\end{thm}

\begin{proof}[Deduction of Theorem \ref{14} from Theorem \ref{13}]  

We highlight this deduction because here is where our result improves on Kowalski and Michel \cite{KM02}. We follow ideas of \cite[proof of Theorem 7.5]{Mon71}.
Our deduction contains three steps.

Step 1: We fix $M$ according to $q,T,R,\ep_0$ as in Theorem \ref{14} and fix $N>M$. In this step, we apply Theorem \ref{13} to obtain an inequality 
\begin{equation}\label{step1}
\int_{-T}^T \sum_{f \in S(q)} \frac{1}{s(f)} \sum_{\substack{\fr \in R(f) \\ \rN\fr \le R}} \frac{1}{|\psi_f(\fr)|} \left| \sum_{\fn} a_\fn \psi_{f,\fr}(\fn)\lambda_f(\fn)\rN\fn^{-it} \right|^2 dt \ll \sum_{\fn} |a_\fn|^2\rN\fn
\end{equation}
for any $T \ge 2$ and any complex numbers $a_\fn$ such that $a_\fn=0$ if $\rN\fn<M$ or $\rN\fn>N$ (so that effectively the inner sum is over $\fn$ with $M \le \rN\fn \le N$). We remark that this type of observation already can be seen in, for example, Theorem 7.1 in \cite{Mon71} and Th\'{e}or\`{e}me 10 in \cite{Bom87}.

Step 2: We replace the integral in (\ref{step1}) by a sum over the imaginary parts $\gamma=\Im(\rho)$ of the well-spaced zeros in $Z(f)$ for each $f 
\in S(q)$. In particular, we show that
\begin{equation}\label{step2}
\sum_{f \in S(q)} \frac{1}{s(f)} \sum_{\rho \in Z(f)} \sum_{\substack{\fr \in R(f) \\ \rN\fr \le R}} \frac{1}{|\psi_f(\fr)|} \left| \sum_{M \le \rN\fn \le N} a_\fn \psi_{f,\fr}(\fn)\lambda_f(\fn)\rN\fn^{-i\gamma} \right|^2  \ll (\log qTN)\sum_{M \le \rN\fn \le N}|a_\fn|^2\rN\fn,
\end{equation}
for any complex numbers $a_\fn$.
 
Step 3: We incorporate the real parts of the well-spaced zeros in (\ref{step2}) via partial summation, ultimately proving (\ref{14eqn}).

\subsection*{Proof of Step 1.}

We recall without proof a lemma of Gallagher; see, e.g., Lemma 1.10 in \cite{Mon71} or Th\'{e}or\`{e}me 9 in \cite{Bom87}. The lemma is essentially an application of Plancherel's theorem.

\begin{lem}\label{C}
If $\displaystyle{S(s)=\sum_{n=1}^\infty b_n n^{-s}}$ is absolutely convergent for $\Re(s) \ge 0$, then for each $T>0$,
\begin{equation*}
\int_{-T}^T |S(it)|^2 dt \ll T^2 \int_0^\infty \left| \sum_{y<n<\tau y} b_n \right|^2 \frac{dy}{y},
\end{equation*} 
where $\tau=e^{\frac{1}{T}}$.
\end{lem}

Given $T \ge 2$, we set $\tau=e^{\frac{1}{T}}$. Observe that $1+\frac{1}{T}<\tau<1+\frac{2}{T}$, which is equivalent to $\frac{T}{2}<(\tau-1)^{-1}<T$. Therefore $2M'<M$, in the notation of Theorems \ref{14} and \ref{13}.
 
Let $a_\fn$ be complex numbers such that $a_\fn$ is nonzero only for $M \le \rN\fn \le N$. In particular, our Dirichlet series below is convergent since it is finite. Apply Lemma \ref{C} to $b_n=\sum_{\rN\fn=n} a_\fn \psi_{f,\fr}(\fn)\lambda_f(\fn)$ and obtain
\begin{eqnarray}\label{mm'}
\nonumber & &  \sum_{f \in S(q)} \frac{1}{s(f)} \sum_{\substack{\fr \in R(f) \\ \rN\fr \le R}} \frac{1}{|\psi_f(\fr)|} \int_{-T}^T \left| \sum_{\fn} a_\fn \psi_{f,\fr}(\fn)\lambda_f(\fn)\rN\fn^{-it} \right|^2 dt \\
& \ll &  \sum_{f \in S(q)} \frac{1}{s(f)} \sum_{\substack{\\ \fr \in R(f) \\ \rN\fr \le R}} \frac{1}{|\psi_f(\fr)|} T^2\int_0^\infty \left| \sum_{y<\rN\fn<\tau y} a_\fn \psi_{f,\fr}(\fn)\lambda_f(\fn) \right|^2 \frac{dy}{y}. 
\end{eqnarray}
Since $a_\fn=0$ for $\rN\fn<M$, the sum inside the integral in (\ref{mm'}) is zero if $y \le \frac{M}{\tau}$. For each $y>\frac{M}{\tau}$, we notice that $y>\frac{M}{\tau}>\frac{M}{2}>M'$ so that for each such fixed $y$ we are able to apply Theorem \ref{13} with $N'=y$, and obtain that the right-hand side in (\ref{mm'}) is
\begin{equation*}
 \ll T^2\int_0^\infty (\tau-1)y \left( \sum_{y<\rN\fn<\tau y} |a_\fn|^2 \right)  \frac{dy}{y}=T^2(\tau-1)\int_0^\infty \left( \sum_{y<\rN\fn<\tau y} |a_\fn|^2 \right) dy. 
\end{equation*}
Interchanging the order of the sum and the integral, we have
\begin{eqnarray*}
\int_0^\infty \left( \sum_{y<\rN\fn<\tau y} |a_\fn|^2 \right) dy &=& \int_0^\infty \sum_{y<n<\tau y}\left(\sum_{\rN\fn=n} |a_\fn|^2\right)dy = \sum_{n=1}^\infty \left( \sum_{\rN\fn=n}|a_\fn|^2 \right) \int_{n/\tau}^n dy \\
&=& \frac{\tau-1}{\tau} \sum_{n=1}^\infty \left( \sum_{\rN\fn=n}|a_\fn|^2 \right) n=\frac{\tau-1}{\tau} \sum_{\fn} |a_\fn|^2\rN\fn,
\end{eqnarray*}
where we recall the last sum is finite.
Therefore, (\ref{mm'}) and hence the left-hand side of (\ref{step1}) is bounded by
\begin{equation}\label{ss1}
 \ll T^2(\tau-1)^2 \tau^{-1} \sum_{\fn} |a_\fn|^2\rN\fn \ll \sum_{\fn} |a_\fn|^2\rN\fn,
\end{equation}
as claimed.

 \subsection*{Proof of Step 2.}
 This step is analogous to the proof of Theorem 7.3 in \cite{Mon71}.
 We recall without proof another lemma of Gallagher; see, e.g., Lemma 1.4 in \cite{Mon71}. 

\begin{lem}\label{A}
Let $T_0,T \ge \delta>0$ be real numbers, and let $\cT$ be a finite set in the interval $[T_0+\frac{\delta}{2},T_0+T-\frac{\delta}{2}]$. Let $S(t)$ be a continuous complex-valued function on the interval $[T_0,T_0+T]$ with continuous derivative in $(T_0,T_0+T)$.
For $x \in \R$, denote
\[ N_\delta(x)=\sum_{\substack{t \in \cT \\ |t-x|<\delta}} 1. \]
Then
\begin{equation}\label{1tool}
\sum_{t \in \cT} N_\delta(t)^{-1} |S(t)|^2 \le \delta^{-1}\int_{T_0}^{T_0+T} |S(t)|^2 dt+\left( \int_{T_0}^{T_0+T} |S(t)|^2 dt \right)^{\frac12} \left( \int_{T_0}^{T_0+T} |S'(t)|^2 dt \right)^{\frac12}.
\end{equation}
\end{lem}

We remark that in Lemma \ref{A}, for any $t \in \cT$, $N_\delta(t)=1$ if $\cT$ is $\delta$-well-spaced, that is, $|t-t'| \ge \delta$ for all $t,t' \in \cT$ such that $t \neq t'$. In this case, Lemma \ref{A} gives a bound for $\sum_{t \in \cT} |S(t)|^2$. Therefore, it is natural to reduce to the case of well-spaced zeros, as we mentioned in (\ref{redubd}) and will derive in Section \ref{redu}. 

Now given $M,N$ as in Theorem \ref{14} and a fixed $f \in S(q)$, $r \in \Z_{>0}$, and 
$t \in \R$, we denote 
\begin{equation*}
\cS(t)=\cS(f,\fr,t)=\sum_{M \le \rN\fn \le N} a_\fn \psi_{f,\fr}(\fn)\lambda_f(\fn)\rN\fn^{-it}.
\end{equation*}
We apply Lemma \ref{A} with $\delta=\eta$, $\cT=Z(f)$, $S(t)=\cS(t)$,
and range of integration $[-T,T]$, to obtain
\begin{equation}\label{sss1}
\sum_{\substack{\rho \in Z(f) \\ \Im(\rho)=\gamma}} |\cS(\gamma)|^2 \le (\log qT)\int_{-T}^T |\cS(t)|^2 dt+\left( \int_{-T}^T |\cS(t)|^2 dt \right)^{1/2} \left( \int_{-T}^T |\cS^{'}(t)|^2 dt \right)^{1/2}.
\end{equation}
The right-hand side of (\ref{sss1}) is
\begin{eqnarray}\label{sst1}
\nonumber &=& (\log qT)\int_{-T}^T |\cS(t)|^2 dt+\left( \log N \int_{-T}^T |\cS(t)|^2 dt \right)^{1/2} \left( \frac{1}{\log N} \int_{-T}^T |\cS'(t)|^2 dt \right)^{1/2} \\
& \le & (\log qT+\log N)\int_{-T}^T |\cS(t)|^2 dt+ \frac{1}{\log N}\int_{-T}^T |\cS^{'}(t)|^2 dt.
\end{eqnarray}
Summing the first integral over $f,r$ and using (\ref{step1}) from Step 1, we have
\begin{equation}\label{stt1}
\sum_{f \in S(q)} \frac{1}{s(f)} \sum_{\substack{\fr \in R(f) \\ \rN\fr \le R}} \frac{1}{|\psi_f(\fr)|} \int_{-T}^T |\cS(t)|^2 dt \ll \sum_{M \le \rN\fn \le N} |a_\fn|^2\rN\fn.
\end{equation}
Then we treat the second integral. Since each summand of $|\cS^{'}(t)|$ only changes that of $|\cS(t)|$ by a multiple of $\log \rN\fn$,
\begin{equation}\label{ttt1}
\sum_{f \in S(q)} \frac{1}{s(f)} \sum_{\substack{\fr \in R(f) \\ \rN\fr \le R}} \frac{1}{|\psi_f(\fr)|} \int_{-T}^T |\cS^{'}(t)|^2 dt \ll \sum_{M \le \rN\fn \le N} |a_\fn\log \rN\fn|^2\rN\fn \le (\log N)^2 \sum_{M \le \rN\fn \le N} |a_\fn|^2\rN\fn.
\end{equation}
Combining (\ref{sst1}), (\ref{stt1}), and (\ref{ttt1}), we have proved (\ref{step2}) of Step 2.

\subsection*{Proof of Step 3.}
First, for each $f \in S(q)$, we replace all real parts of the well-spaced zeros in $Z(f) \subseteq M(\alpha,T)$ by $\alpha$. For fixed $M$ and $\rho \in \C$, we define the sum
\begin{equation*}
S(\rho,u)=\sum_{M \le \rN\fn \le u} a_\fn\psi_{f,\fr}(\fn)\lambda_f(\fn)\rN\fn^{-\rho}.
\end{equation*}
For any fixed $N>M$ and $\rho=\beta+i\gamma \in M(\alpha,T)$ (so that $\beta \ge \alpha$), we have by partial summation
\begin{eqnarray*}
S(\beta+i\gamma,N) &=& 
S(\alpha+i\gamma,N)N^{\alpha-\beta}+(\beta-\alpha) \int_{M}^{N}S(\alpha+i\gamma,u)u^{\alpha-\beta-1} du.
\end{eqnarray*}
Then regardless of $\beta>\alpha$ or $\beta=\alpha$, we have
\begin{eqnarray*}
\left| S(\beta+i\gamma,N) \right|^2 & \ll & |S(\alpha+i\gamma,N)|^2 + (\beta-\alpha)^2 \left( \int_{M}^{N} \frac{\log u}{u^{2\beta-2\alpha+1}} du \right) \left(\int_{M}^{N} \frac{|S(\alpha+i\gamma,u)|^2}{u\log u} du \right) \\
& \ll & |S(\alpha+i\gamma,N)|^2+\int_{M}^{N} \frac{|S(\alpha+i\gamma,u)|^2}{u\log u} du.
\end{eqnarray*}
Second, we apply (\ref{step2}) from Step 2, where $a_n$ is replaced by $a_nn^{-\alpha}$. Then from the above inequality,
\begin{eqnarray*}
& & \sum_{f \in S(q)} \frac{1}{s(f)} \sum_{\rho \in Z(f)} \sum_{\substack{\fr \in R(f) \\ \rN\fr \le R}} \frac{1}{|\psi_f(\fr)|} \left| S(\beta+i\gamma,N) \right|^2 \\
& \ll & \sum_{f \in S(q)} \frac{1}{s(f)} \sum_{\rho \in Z(f)} \sum_{\substack{\fr \in R(f) \\ \rN\fr \le R}} \frac{1}{|\psi_f(\fr)|} \left[ |S(\alpha+i\gamma,N)|^2+\int_{M}^{N} \frac{|S(\alpha+i\gamma,u)|^2}{u\log u} du \right] \\
& \ll & (\log qTN)\sum_{M \le \rN\fn \le N}|a_\fn|^2\rN\fn^{1-2\alpha}+(\log qTN)\int_{M}^{N} \sum_{M \le \rN\fn \le u} |a_\fn|^2 \rN\fn^{1-2\alpha} \frac{1}{u\log u} du.
\end{eqnarray*}
Note that 
\begin{eqnarray}\label{mqtr}
\nonumber \int_{M}^{N} \sum_{M \le \rN\fn \le u} |a_\fn|^2 \rN\fn^{1-2\alpha} \frac{1}{u\log u} du &=& \sum_{M \le \rN\fn \le N} |a_\fn|^2\rN\fn^{1-2\alpha} \int_n^{N} \frac{1}{u\log u} du \\
&=& \sum_{M \le \rN\fn \le N} |a_\fn|^2\rN\fn^{1-2\alpha} \log \frac{\log N}{\log n},
 \end{eqnarray}
 and in (\ref{mqtr}), 
 \[ \log \frac{\log N}{\log n} \ll \log \frac{\log N}{\log qTR} \]
 since $\log M \gg \log qTR$ by (\ref{3.10}).
In conclusion, Theorem \ref{14} holds, once Theorem \ref{13} is proved.
\end{proof}

\end{section}

\begin{section}{Proof of Theorem \ref{13}}\label{notations}

\begin{subsection}{Key lemmas to prove Theorem \ref{13}}\label{lemsfor13}

Now we state the necessary lemmas to prove Theorem \ref{13}, deferring their proofs to Section \ref{pf13lem}. First, we show the convexity bound for $L(f \times g,s)$ and a bound for $\mathrm{Cond}(f \times g)$, for any $f,g \in S(q)$.

\begin{lem}\label{10}
Fix $q \ge 1$, let $f,g \in S(q)$ be two cuspidal automorphic representations of $\mathrm{GL}_n(\mathbb{A}_k)$, and let $s=\sigma+it$. Assume that the Ramanujan-Petersson conjecture and a uniform bound on Langlands parameters at Archimedean places. Then we have
\begin{equation}\label{convf}
L(f,s) \ll_\ep (\mathrm{Cond}(f)(|t|+2)^{nn_k})^{\frac{1-\sigma}{2}+\ep}
\end{equation}
for $0 \le \sigma \le 1$ and any $\ep>0$, and
\begin{equation}\label{102}
L(f \times g,s) \ll_\ep (\mathrm{Cond}(f \times g)(|t|+2)^{n^2n_k})^{\frac{1-\sigma}{2}+\ep}
\end{equation}
for $0 \le \sigma \le 1$ and any $\ep>0$.
Moreover, from \cite[Theorem 1]{BH97}, we have
\begin{equation}\label{bh}
\mathrm{Cond}(f \times g) \le (\mathrm{Cond}(f)\mathrm{Cond}(g))^n.
\end{equation}
\end{lem}

We recall the definition of $\psi_f(\fr)$ in (\ref{psidef}) and recall also the parameter $A$ such that $\mathrm{Cond}(f) \le q^A$ for all $f \in S(q)$. 
The next lemma controls the average size of $\frac{1}{|\psi_f(\fr)|}$.

\begin{lem}\label{11}
Let $f \in S(q)$. Then for $R>q^C$, where $C$ is any constant with $C>nA$,
\begin{equation}\label{111}
\sum_{\substack{\fr \in R(f) \\ \rN\fr \le R }} \frac{1}{|\psi_f(\fr)|} \gg s(f)\log R.
\end{equation}
 
Moreover, for any $R \ge 2$,
\begin{equation}\label{112}
\sum_{\substack{\fr \in R(f) \\ \rN\fr \le R }} \frac{1}{|\psi_f(\fr)|} \ll s(f)\log R.
\end{equation}
\end{lem}

We factorize our $L$-functions so that the $L$-functions behave as those of degree 1.

\begin{lem}\label{9}
We fix $z \ge 1$, set $\displaystyle{\fP=\prod_{\rN\fp<z} \fp}$ and let $f,g \in S(q)$. Then we have
\begin{equation}\label{factfg}
L^{\mathrm{ur}}(f,s)=L^\flat(f,s)L^\sharp(f,s),
 \ \ \ \ L^{\mathrm{ur}}(f \times g,s)=L^\flat(f \times g,s)L^\sharp(f \times g,s)
\end{equation}
with
\begin{equation}\label{2.5}
L^\flat (f,s)=\sideset{}{_{}^{\flat}}\sum_{(\fn,\fP)=1} \lambda_{f}(\fn)\rN\fn^{-s}=\prod_{\substack{\rN\fp \ge z}}(1+\lambda_{f}(\fp)\rN\fp^{-s})
\end{equation}
and
\begin{equation*}
L^\flat (f \times g,s)=\sideset{}{_{}^{\flat}}\sum_{(\fn,\fP)=1} \lambda_{f \times g}(\fn)\rN\fn^{-s}=\prod_{\substack{\rN\fp \ge z}}(1+\lambda_f(\fp)\lambda_g(\fp)\rN\fp^{-s}).
\end{equation*}
Here the notation $\sum^\flat$ denotes a sum over squarefree ideals, and the notation $L^\flat$ denotes that the Dirichlet series of the $L$-function only involves squarefree ideals. 
The function $L^\sharp(f,s)$ (resp. $L^\sharp(f\times g,s)$) is holomorphic and has neither zero nor pole in $\Re(s)>\frac12$. Moreover, $L^\sharp(f,s)$ (resp. $L^\sharp(f\times g,s)$) is uniformly bounded in the region $\Re(s)>\frac12+\ep$ for any fixed $\ep>0$. 
\end{lem}

Let $f,g \in S(q)$. We recall the definition of $R(f)$ (a set of squarefree ideals of $\cO_k$) in (\ref{rfdef}) and let $\fr \in R(f)$, $\ft \in R(g)$. Let a set of coefficients $\{h(\fd)\}_{\fd}$ depending on $f,g$ and $\fr,\ft$ be defined by
\begin{equation}\label{3.6}
\sum_{\fd} h(\fd)\rN\fd^{-s}=\prod_{\substack{\fp|\fr \\ \fp \nmid \ft}}(1+(\psi_f(\fp)-1)\rN\fp^{-s}) \prod_{\substack{\fp|\ft \\ \fp \nmid \fr}}(1+(\psi_g(\fp)-1)\rN\fp^{-s}) \prod_{\fp|(\fr,\ft)}(1+(\psi_f(\fp)\psi_g(\fp)-1)\rN\fp^{-s}).
\end{equation}
By construction, the function $h(\fd)$ is supported on the squarefree ideals $\fd|\fr\ft$. In closed form, we can compute
\begin{equation}\label{hd}
h(\fd)=\prod_{\substack{\fp|\fd \\ \fp|\fr \\ \fp \nmid \ft}} (\psi_f(\fp)-1) \prod_{\substack{\fp|\fd \\ \fp|\ft \\ \fp \nmid \fr}} (\psi_g(\fp)-1) \prod_{\substack{\fp|\fd \\ \fp|(\fr,
\ft)}} (\psi_f(\fp)\psi_g(\fp)-1).
\end{equation}
Since $\frac{\rN\fp}{|\lambda_f(\fp)|^2} \ge \rN\fp^{1-2\delta}>1$ for $\fp \in R(f)$, note that
\begin{equation}\label{hdbd}
|h(\fd)| \le \left( \prod_{\fp|\fr} 2|\psi_f(\fp)| \right) \left( \prod_{\fp|\ft} 2|\psi_g(\fp)| \right).
\end{equation}

Recall the pseudo-characters $\psi_{f,\fr}(\fn)$ defined in (\ref{psidef}).
The next lemma shows orthogonality (cancellation) among pseudo-characters.

\begin{lem}\label{12}
Let $f,g,\fr,\ft,\{h(\fd)\}_{\fd}$ be as above.
We have
\begin{equation}\label{3.7}
\psi_{f,\fr}(\fn)\psi_{g,\ft}(\fn)=\mu_k(\fn)^2\sum_{\fd|\fn}h(\fd)
\end{equation}
for all $\fn$. If $g=\overline{f}$, we have
\begin{equation}\label{3.8}
\sum_{\fd} h(\fd)|\lambda_f(\fd)|^2\rho_f(\fd)\rN\fd^{-1}=\delta(\fr,\ft)|\psi_f(\fr)|,
\end{equation}
where
\begin{equation}\label{defofrho}
\rho_f(\fd):=\prod_{\fp|\fd} (1+|\lambda_f(\fp)|^2\rN\fp^{-1})^{-1}
\end{equation}
and the function $\delta(\fr,\ft)$ is $1$ if $\fr=\ft$ and is $0$ otherwise.
\end{lem}

\end{subsection}

We derive Theorem \ref{13} from the above lemmas in Section \ref{pf13} and prove the lemmas in Section \ref{pf13lem}.

\begin{subsection}{Proof of Theorem \ref{13}}\label{pf13}
We essentially follow the strategy of Proposition 13 in \cite{KM02}, making necessary changes. We fix $q,R,\ep_0,\tau$ and fix $N'$ satisfying (\ref{3.100}) in the hypothesis of Theorem \ref{13}. We proceed in 6 steps.

\subsubsection*{Step 1. Translating to the dual problem.}

In fact, we will prove a dual statement to Theorem \ref{13}; to see why this suffices, we first make a general observation.

Let $\Theta, \Gamma$ be two countable index sets. For any $\{a_\fn\}_{\fn \in \Theta} \in \ell^2$ 
and for fixed constants $c_{\fn,\gamma}$ with $\fn \in \Theta, \gamma \in \Gamma$ such that
\begin{equation}\label{cng}
\sum_{\fn,\gamma}|c_{\fn,\gamma}|^2<\infty,
\end{equation} we define a linear operator 
\begin{eqnarray*}
\nonumber A: \ell^2 & \to & (\ell^2)^\ast \cong \ell^2 \\
\{a_\fn\}_{\fn \in \Theta} & \mapsto & \{ \sum_{\fn} c_{\fn,\gamma}a_\fn\}_{\gamma \in \Gamma}.
\end{eqnarray*}
The inequality
\begin{equation*}
\sum_{\gamma \in \Gamma} \left| \sum_\fn c_{\fn,\gamma}a_\fn \right|^2 \le \sum_{\gamma \in \Gamma} \left( \sum_\fn |c_{\fn,\gamma}|^2 \right) \left(\sum_\fn |a_\fn|^2\right)=\left(\sum_\fn |a_\fn|^2\right) \left( \sum_{\fn,\gamma}|c_{\fn,\gamma}|^2 \right) 
\end{equation*}
and the condition (\ref{cng}) ensure that $A$ is bounded as an operator, with norm
\[ \|A\| \le \left( \sum_{\fn,\gamma}|c_{\fn,\gamma}|^2 \right)^{\frac12}.\]
The dual operator $A^\ast$ is
\begin{eqnarray*}
\nonumber A^\ast: (\ell^2)^\ast \cong \ell^2 & \to &  \ell^2 \\
\{ b_{\gamma}\}_{\gamma \in \Gamma} & \mapsto & \{ \sum_{\gamma \in \Gamma} c_{\fn,\gamma}b_{\gamma}\}_{\fn},
\end{eqnarray*}
and $\| A^\ast \|=\| A \|$. In some settings, like ours, it is easier to bound the dual operator $A^\ast$.

We apply this now with the index set $\Gamma$ being the set
$\{ \{f,\fr\}: f \in S(q), \fr \in R(f), \rN\fr \le R\}$, and choose 
\begin{equation*}
c_{\fn,\gamma}=c_{\fn,\{f,\fr\}}=\frac{1}{\sqrt{s(f)|\psi_f(\fr)|}}\psi_{f,\fr}(\fn)\lambda_f(\fn).
\end{equation*}
Defining the operator $A$ as above, since $\| A^\ast \|=\| A \|$, we conclude that proving Theorem \ref{13} is equivalent to proving the dual statement (i.e., bounding $\|A^\ast\|$), namely that 
\begin{equation}\label{smooth}
 \sum_{N'<\rN\fn \le \tau N'} \left| \sum_{f \in S(q)} \frac{1}{\sqrt{s(f)}} \sum_{\substack{\fr \in R(f) \\ \rN\fr \le R}} \frac{b(f,\fr)}{\sqrt{|\psi_f(\fr)|}} \lambda_f(\fn) \psi_{f,\fr}(\fn) \right|^2 \ll (\tau-1)N'\sum_{\substack{f \in S(q) \\ \fr \in R(f) \\ \rN\fr \le R}}  |b(f,\fr)|^2 
\end{equation}
for any complex numbers $b(f,\fr)$. It is more advantageous to prove (\ref{smooth}) than to prove (\ref{13eqn}) directly because we are able to take advantage of summing over $f,\fr$ inside the absolute value in order to exploit orthogonality (cancellation) among pseudo-characters in (\ref{smooth}), via an application of Lemma \ref{12}.

\subsubsection*{Step 2. Shifting the contour.}

Recall that $1<\tau \le 2$ is fixed in Theorem \ref{13} and $z,\fP$ are as in Section \ref{secofls}.
We choose a smooth test function $\varphi: [0,\infty) \to [0,1]$ with compact support in $[1/2,3]$ such that $\varphi(x)=1$ for $1 \le x \le \tau$. Then the left-hand side of (\ref{smooth}) is bounded above by
\begin{equation}\label{giant}
\sum_{\fn} \varphi(\frac{\rN\fn}{N'}) \left| \sum_{f \in S(q)} \frac{1}{\sqrt{s(f)}} \sum_{\substack{\fr \in R(f) \\ \rN\fr \le R}} \frac{b(f,\fr)}{\sqrt{|\psi_f(\fr)|}} \lambda_f(\fn) \psi_{f,\fr}(\fn) \right|^2 = \sum_{f,g \in S(q)} \sum_{\substack{\fr \in R(f) \\ \ft \in R(g) \\ \rN\fr,\rN\ft \le R}} \frac{b(f,\fr)\overline{b(g,\ft)}}{\sqrt{s(f)s(g)|\psi_f(\fr)\psi_g(\ft)|}} S_1
\end{equation}
in which we define
\begin{equation}\label{3.14}
S_1=S_1(f,g,\fr,\ft)=\sum_{\fn} \psi_{f,\fr}(\fn)\overline{\psi_{g,\ft}(\fn)}\lambda_f(\fn)\overline{\lambda_g(\fn)}\varphi(\frac{\rN\fn}{N'}).
\end{equation}
For each such $f,g,\fr,\ft$ define $\{h(\fd)\}_{\fd}$ as in (\ref{3.6}). We will suppress the notational dependence of $S_1$ and $h$ on $f,g,\fr,\ft$.
By a key identity (\ref{3.7}) in Lemma \ref{12} and after rewriting the double sum, we have
\begin{equation}\label{3.14sec}
S_1 = \sideset{}{_{}^{\flat}}\sum_{\substack{(\fd,\fP)=1}} h(\fd)\lambda_f(\fd) \overline{\lambda_g(\fd)} T_\fd(N')
\end{equation}
in which we define
\begin{equation*}
T_\fd(N')=\sideset{}{_{}^{\flat}}\sum_{\substack{(\fn,\fd)=1 \\ (\fn,\fP)=1}} \lambda_f(\fn) \overline{\lambda_g(\fn)}\varphi(\frac{\rN(\fn\fd)}{N'}).
\end{equation*}
In both sums, recall that $\sideset{}{_{}^{\flat}}\sum$ denotes a summation over squarefree ideals.
For each $\fd$, we define
\begin{equation}\label{ldandl}
 L_\fd^\flat (f \times \overline{g},s) = \sideset{}{_{}^{\flat}}\sum_{\substack{(\fn,\fd)=1 \\ (\fn,\fP)=1}} \lambda_f(\fn) \overline{\lambda_g(\fn)} \rN\fn^{-s} = L^\flat(f \times \overline{g},s) \prod_{\substack{\fp|\fd \\ \rN\fp \ge z}} (1+\lambda_f(\fp)\overline{\lambda_g(\fp)}\rN\fp^{-s})^{-1},
\end{equation}
where $L^\flat(f \times \overline{g},s)$ is defined in Lemma \ref{9}. 
By Mellin inversion we can write
\begin{equation*}
T_\fd(N')=\frac{1}{2\pi i} \int_{(3)} L_\fd^\flat (f \times \overline{g},s) (\cM{\varphi})(s)(N'/\rN\fd)^s ds.
\end{equation*}
For convenience, define
\begin{equation}\label{hds}
H_\fd(s):=L_\fd^\flat (f \times \overline{g},s) (\cM{\varphi})(s)(N'/\rN\fd)^s.
\end{equation}
We truncate the integral to height $|t|=T'$, where $T'$ is a sufficiently large number (relative to $f,g,q,\tau,R,\varphi,\ep_0,\fd$). Then we move the line of integration to $\Re(s)=\frac12+\ep, |\Im(s)| \le T'$, for any fixed $\ep>0$. 
By Lemma \ref{9} and Remark \ref{rm1}, $H_\fd(s)$ has a possible pole at $s=1$, in precisely the case that $f=g$. Then 
\begin{equation}\label{tdn}
 T_\fd(N')=\frac{1}{2\pi i} \int_{(3)} H_\fd(s) ds = \mathrm{Res}_{s=1}H_\fd(s)+E_{\fd,1}+E_{\fd,2}
 \end{equation}
 where
 \begin{equation*}
 E_{\fd,1}=\frac{1}{2\pi i} \int_{\frac12+\ep-iT'}^{\frac12+\ep+iT'} H_\fd(s) ds
 \end{equation*}
 and
 \begin{equation}\label{ed2exp}
 E_{\fd,2}=\frac{1}{2\pi i} \left( \int_{\frac12+\ep+iT'}^{3+iT'} - \int_{\frac12+\ep-iT'}^{3-iT'} + \int_{3-i\infty}^{3-iT'} + \int_{3+iT'}^{3+i\infty} \right) H_\fd(s) ds.
 \end{equation}
 We will show that the contribution of $\displaystyle{\mathrm{Res}_{s=1}H_\fd(s)}$ to $S_1$ is precisely 
 \begin{equation}\label{rescontr}
 s(f)L^\sharp(f \times \overline{f},1)(\tau-1)N'\delta(\fr,\ft)|\psi_f(\fr)|,
 \end{equation}
  while 
  \begin{equation}\label{errore1e2}
  E_{\fd,1}=O(N'^{1/2+\ep}q^{nA/2}R^{1+2\delta+3\ep}), \ \ E_{\fd,2}=o(1).
  \end{equation}
 Substituting these contributions for each $S_1$ in (\ref{giant}), we will conclude that the main term in (\ref{giant}) is 
 \[ N'(\tau-1)\sum_{f \in S(q)} \sum_{\substack{\fr \in R(f) \\ \rN\fr \le R}} |b(\fr,f)|^2 \]
  and it dominates the error terms under the hypothesis (\ref{3.100}) of the theorem. Theorem \ref{13} then holds, once we have proved (\ref{rescontr}) and (\ref{errore1e2}). We compute the main term in Step 3, the error terms in Steps 4,5, and substitute the contributions in Step 6.
 
 \subsubsection*{Step 3. Main term in $S_1$.}
 
For $f,g \in S(q)$, we define the function $\delta(f,g)$ such that it is $1$ if $f=g$ and is $0$ otherwise.
 
By Remark \ref{rm1}, $L(f \times \overline{g},s)$ only has at most one pole, which is at $s=1$ and occurs if and only if $f=g$. Recall the definition of $\rho_f(\fd)$ in (\ref{defofrho}). Therefore,
\begin{equation*}
\mathrm{Res}_{s=1}H_\fd(s)=\delta(f,g)s(f)\rho_f(\fd)L^\sharp(f \times \overline{f},1)(\tau-1)N'\rN\fd^{-1}.
\end{equation*}
Summing over $\fd$,
\begin{equation}\label{drt}
\sideset{}{_{}^{\flat}}\sum_{\substack{(\fd,\fP)=1}} h(\fd)\lambda_f(\fd) \overline{\lambda_g(\fd)} \mathrm{Res}_{s=1}H_\fd(s) =s(f)L^\sharp(f \times \overline{f},1)(\tau-1)N'\sideset{}{_{}^{\flat}}\sum_{\substack{(\fd,\fP)=1}} h(\fd)\rho_f(\fd)|\lambda_f(\fd)|^2 \rN\fd^{-1}. 
\end{equation}
We recall from (\ref{hd}) that $h(\fd)$ is defined according to the fixed $\fr,\ft$ chosen in (\ref{giant}) and is supported on $\fd|\fr\ft$, where $\fd$ is squarefree. With the identity (\ref{3.8}) in Lemma \ref{12} and the relation $N'>R^2 \ge \rN(\fr\ft)$ by the hypothesis (\ref{3.100}), the right-hand side of (\ref{drt}) is equal to
\begin{equation*}
s(f)L^\sharp(f \times \overline{f},1)(\tau-1)N'\delta(\fr,\ft)|\psi_f(\fr)|;
\end{equation*}
recall the definition of $\psi_f(\fr)$ from (\ref{psidef}).
We also note that by Lemma \ref{9},
\begin{equation}\label{lsharp}
L^\sharp(f \times \overline{f},s) \asymp_\ep 1
\end{equation}
for any $\Re(s)>\frac12+\ep$ with any $\ep>0$. Thus, in $S_1$, the main term is 
\begin{equation*}
\asymp s(f)(\tau-1)N'\delta(\fr,\ft)|\psi_f(\fr)|.
\end{equation*}

\subsubsection*{Step 4. Contribution from $E_{\fd,2}$ to $S_1(f,g,\fr,\ft)$.}

For the error terms, we estimate $H_\fd(s)$ by showing that $\cM{\varphi}(s)$ decays faster than the the growth of $L_\fd^\flat(f \times \overline{g},s)$. On the one hand, for $s$ such that $\Re(s)$ is not a non-positive integer, integration by parts $m$ times shows
\begin{equation}\label{ms}
|\cM{\varphi}(s)| = \left| \frac{1}{s(s+1) \dots (s+m)} \int_{\frac12}^3 \varphi^{(m+1)}(t) t^{s+m} dt \right| 
\end{equation}
for any positive integer $m$. Thus, $\cM{\varphi}(s)$ decays faster than $|s|^{-m}$, for any positive integer $m$. This fast decay allows us to find a large real number $T'$ (relative to $f,g,q,\tau,R,\varphi,\ep_0,\fd$) such that for the horizontal integrals of $E_{\fd,2}$, we can find an $m$ large enough such that
\begin{equation}\label{mpbd}
|\cM\varphi(s)| =o\left(R^{-3}(\mathrm{Cond}(f \times \overline{g}){T'}^{\frac{n^2n_k}{4}}{N'}^3)^{-1} \right)
\end{equation}
for $|\Im(s)|=T', \frac12+\ep \le \Re(s) \le 3$, and that for the vertical integrals of $E_{\fd,2}$, the last two terms in (\ref{ed2exp}) are $=o(R^{-3})$. 

On the other hand, we estimate $L_\fd^\flat (f \times \overline{g},s)$ using the factorizations (\ref{ldandl}), (\ref{factfg}), and (\ref{raun}). 
Here we will learn that we must require $z$ to satisfy
\begin{equation}\label{zfirst}
z \ge n^4n_k^2.
\end{equation}
Recalling that $|\lambda_f(\fp)| \le n$ for any $\fp$ when $f \in S(q)$, with $z$ as above, we have 
\begin{equation*}
\prod_{\substack{\fp|\fd \\ \rN\fp \ge z}} (1+\lambda_f(\fp)\overline{\lambda_g(\fp)}\rN\fp^{-s})^{-1} \ll 2^{\omega(\fd)} \ll_\ep \rN\fd^\ep,
\end{equation*}
 where $\omega(\fd)$ is the number of prime ideal divisors of $\fd$. Thus,
\begin{equation}\label{ldl}
L_\fd^\flat (f \times \overline{g},s) \ll_\ep  \rN\fd^\ep L^\flat(f \times \overline{g},s).
\end{equation}
By (\ref{lsharp}), we have
\begin{equation}\label{ldl2}
L^\flat (f \times \overline{g},s) \asymp_\ep L^{\mathrm{ur}}(f \times \overline{g},s) \asymp_\ep L(f \times \overline{g},s)
\end{equation}
for $\Re(s) \ge \frac12+\ep$.
We apply the convexity bound of Lemma \ref{10} to $L(f \times \overline{g},s)$ and combine (\ref{ldl}), (\ref{ldl2}) so that
\begin{equation}\label{ldbd}
L_\fd^\flat (f \times \overline{g},s) \ll_\ep \rN\fd^\ep \mathrm{Cond}(f \times \overline{g})|\Im(s)|^{\frac{n^2n_k}{4}}.
\end{equation}
Applying (\ref{mpbd}) and (\ref{ldbd}) in the definition of $H_\fd(s)$, we have $|H_\fd(s)|=o(\rN\fd^\ep R^{-3})$ for $|\Im(s)|=T', \frac12+\ep \le \Re(s) \le 3$. Thus,
\begin{equation}\label{e21}
|E_{\fd,2}| =o(\rN\fd^\ep R^{-3}).
\end{equation}

We now compute the contribution of $E_{\fd,2}$ to $S_1$.
By (\ref{hdbd}), and the fact that $h(\fd)$ is supported on squarefree divisors of $\fr\ft$, we have 
\begin{eqnarray}\label{e22}
& & \sideset{}{_{}^{\flat}}\sum_{\substack{(\fd,\fP)=1}} h(\fd)\lambda_f(\fd) \overline{\lambda_g(\fd)}\rN\fd^\ep \ll \prod_{\fp|\fr}(1+2|\psi_f(\fp)\lambda_f(\fp)|\rN\fp^\ep)\prod_{\fp|\ft}(1+2|\psi_g(\fp)\lambda_g(\fp)|\rN\fp^\ep) \\
\nonumber & \ll &  R^{2+2\delta+3\ep}.
\end{eqnarray}
The last inequality holds because we have $|\lambda_f(\fp)|>\rN\fp^{-\delta}$ for any $\fp|\fr$ with $\fr \in R(f)$, and thus by definition (\ref{psidef}), $|\psi_f(\fp)\lambda_f(\fp)| \le \rN\fp^{1+\delta}$ (and similarly for the other factor). Taking (\ref{e21}) and (\ref{e22}) into (\ref{3.14sec}), we know that the contribution to $S_1$ from $E_{\fd,2}$ is $o(R^{-1+2\delta+3\ep})$, which is at most $o(1)$.

\subsubsection*{Step 5. Contribution from $E_{\fd,1}$ to $S_1(f,g,\fr,\ft)$.}

On $\Re(s)=1/2+\ep$, we use (\ref{ldbd}) and (\ref{bh}) (also recall the parameter $A$ from the family $S(q)$) to obtain
\begin{equation*}
L_\fd^\flat(f \times \overline{g},s) \ll_\ep q^{nA/2}\rN\fd^\ep |\Im(s)|^{\frac{n^2n_k}{4}}.
\end{equation*} 

Note that $\cM\varphi(s)$ has fast decay as shown in (\ref{ms}).  It follows that
 \begin{equation*}
 E_{\fd,1}=O(N'^{1/2+\ep}\rN\fd^{-1/2+\ep}q^{nA/2}\int_{\frac12+\ep-iT'}^{\frac12+\ep+iT'}|\cM\varphi(s)|) =O(N'^{1/2+\ep}\rN\fd^{-1/2+\ep}q^{nA/2}).
\end{equation*}

We sum over $d$ and obtain
\begin{equation*}
\sideset{}{_{}^{\flat}}\sum_{\substack{(\fd,\fP)=1}} h(\fd)\lambda_f(\fd) \overline{\lambda_g(\fd)} E_{\fd,1} \ll_\ep N'^{1/2+\ep}q^{nA/2}\sideset{}{_{}^{\flat}}\sum_{\substack{(\fd,\fP)=1}} |h(\fd)\lambda_f(\fd)\lambda_g(\fd)|\rN\fd^{-1/2+\ep}.
\end{equation*}
Similar to (\ref{e22}), we conclude
\begin{equation*}
\sideset{}{_{}^{\flat}}\sum_{\substack{(\fd,\fP)=1}} \frac{|h(\fd)\lambda_f(\fd)\lambda_g(\fd)|}{\rN\fd^{1/2-\ep}} \ll \prod_{\fp|\fr}(1+\frac{2|\psi_f(\fp)\lambda_f(\fp)|}{\rN\fp^{1/2-\ep}}) \prod_{\fp|\ft}(1+\frac{2|\psi_g(\fp)\lambda_g(\fp)|}{\rN\fp^{1/2-\ep}}) \ll_\ep R^{1+2\delta+3\ep}.
\end{equation*}

\subsubsection*{Step 6. Assembling all terms.}

Now we finally sum all three types of contribution to $S_1=S_1(f,g,\fr,\ft)$.
From Step 3, the main term in (\ref{giant}) is equal to 
\begin{equation}\label{wtmot}
\sum_{f \in S(q)}\sum_{\substack{\fr \in R(f) \\ \rN\fr \le R}} \frac{|b(\fr,f)|^2}{s(f)|\psi_f(\fr)|} s(f)L^\sharp(f \times \overline{f},1)(\tau-1)N'|\psi_f(\fr)| \asymp N'(\tau-1)\sum_{f \in S(q)} \sum_{\substack{\fr \in R(f) \\ \rN\fr \le R}} |b(\fr,f)|^2.
\end{equation}
The cancellation of the factors $s(f)|\psi_f(\fr)|$ in (\ref{wtmot}) motivates the large sieve weighting we chose in Theorem \ref{13} (hence also in Theorem \ref{14}).

From Steps 4,5, the error term $E_{\fd,1}$ dominates $E_{\fd,2}$. In (\ref{giant}), the contribution from $E_{\fd,1}$ is at most
\begin{eqnarray*}
\nonumber & \ll_\ep & \sum_{f,g \in S(q)} \sum_{\substack{\fr \in R(f) \\ \ft \in R(g) \\ \rN\fr, \rN\ft \le R}} \frac{|b(f,\fr)\overline{b(g,\ft)}|}{\sqrt{s(f)s(g)|\psi_f(\fr)\psi_g(\ft)|}} R^{1+2\delta+3\ep}N'^{1/2+\ep}q^{nA/2} \\
&=& R^{1+2\delta+3\ep}N'^{1/2+\ep}q^{nA/2}\left( \sum_{f \in S(q)} \frac{1}{\sqrt{s(f)}} \sum_{\substack{\fr \in R(f) \\ \rN\fr \le R}} \frac{|b(f,\fr)|}{\sqrt{|\psi_f(\fr)|}} \right)^2.
\end{eqnarray*}
By Cauchy's inequality, the squared expression is at most
\begin{equation*}
 \le  \left(\sum_{f \in S(q)} \sum_{\substack{\fr \in R(f) \\ \rN\fr \le R}} |b(f,\fr)|^2 \right) \left( \sum_{f \in S(q)} \frac{1}{s(f)} \sum_{\substack{\fr \in R(f) \\ \rN\fr \le R}} \frac{1}{|\psi_f(\fr)|} \right) \ll |S(q)|(\log R) \sum_{f \in S(q)} \sum_{\substack{\fr \in R(f) \\ \rN\fr \le R}} |b(f,\fr)|^2,
\end{equation*}
in which we applied (\ref{112}).
In conclusion, the left-hand side of (\ref{smooth}) is
\begin{eqnarray}\label{smer}
 & \le & \sum_{f,g \in S(q)} \sum_{\substack{\fr \in R(f) \\ \ft \in R(g) \\ \rN\fr, \rN\ft \le R}} \frac{b(f,\fr)\overline{b(g,\ft)}}{\sqrt{s(f)s(g)|\psi_f(\fr)\psi_g(\ft)|}} S_1 \\
\nonumber & \ll_\ep & \left( N'(\tau-1)+N'^{1/2+\ep}q^{nA/2+d}R^{1+2\delta+3\ep}(\log R) \right) \sum_{f \in S(q)} \sum_{\substack{\fr \in R(f) \\ \rN\fr \le R}} |b(\fr,f)|^2.
\end{eqnarray}
Here we have used the fact that $|S(q)| \ll q^d$. Since we can choose $\ep$ arbitrarily small, by (\ref{3.100}),
\begin{equation}\label{smer31}
N'^{1/2+\ep}q^{nA/2+d}R^{1+2\delta+3\ep}(\log R) < N'(\tau-1).
\end{equation}
The inequalities (\ref{smer}) and (\ref{smer31}) together prove (\ref{smooth}), hence the theorem.
\end{subsection}

\begin{subsection}{Proof of key lemmas}\label{pf13lem}

To complete the proof of Theorem \ref{13}, it remains to prove the lemmas stated in Section \ref{lemsfor13}.

\begin{proof}[Proof of Lemma \ref{10}]
We only prove (\ref{102}). 

 The functional equation (\ref{fedifform}) is equivalent to
\begin{equation}\label{fe}
 \ep(f \times g)^{-1}(D_k^{n^2}\rN\fq_{f \times g})^{s-\frac12} \prod_{v \in S_\infty} L(f_v \times g_v,s) L(f_v \times g_v,1-s)^{-1} L(f \times g,s)=L(\overline{f \times g},1-s).
\end{equation}
Noting that $L(\overline{f \times g},1-s)=\overline{L(f\times g,1-\overline{s})}$, we have
\begin{equation*}
\left|  (D_k^{n^2}\rN\fq_{f \times g})^{s-\frac12} \prod_{v \in S_\infty} L(f_v \times g_v,s) L(f_v \times g_v,1-s)^{-1} L(f \times g,s) \right|=\left| L(f \times g,1-\overline{s}) \right|.
\end{equation*}
Using the equation (5.115) in \cite{IK04} (to bound the quotient of gamma factors) and letting $s \to 1^{+}$ in the equation above, we have
\begin{equation}\label{plprin2}
L(f \times g,-\ep) \ll_\ep \mathrm{Cond}(f \times g)^{\frac12+\ep}.
\end{equation}
Note that (\ref{plprin2}) also holds for a twist $(f \times g) \otimes |\mathrm{det}|^{it}$ for any $t \in \R$; see, e.g., (1.11) in \cite{Har03}. Thus,
\begin{equation}\label{plprin3}
L(f \times g,-\ep+it) \ll_\ep [\mathrm{Cond}(f \times g)(|t|+2)^{n^2n_k}]^{\frac12+\ep}.
\end{equation}
The convexity bound
\begin{equation*}
L(f \times g,s) \ll_\ep (\mathrm{Cond}(f \times g)(|t|+2)^{n^2n_k})^{\frac{1-\sigma}{2}+\ep}
\end{equation*}
then follows from the Phragmen-Lindel\"{o}f principle applied to (\ref{plprin3}) and the fact that $L(f\times g,1+it) \ll_\ep 1$.
\end{proof}

\begin{proof}[Proof of Lemma \ref{11}]
We write
\begin{equation*}
\sum_{\substack{\fr \in R(f) \\ \rN\fr \le R}} \frac{1}{|\psi_f(\fr)|}= \sum_{\substack{\fr \in R(f) \\ \rN\fr \le R}} \frac{|\lambda_f(\fr)|^2}{\rN\fr}=\sum_{\substack{\fr \in R(f) \\ \rN\fr \le R}} \frac{\lambda_{f \times \overline{f}}(\fr)}{\rN\fr}.
\end{equation*}
Let $\varphi$ be a smooth function: $[0,\infty) \to [0,1]$ with compact support in $[1/2,3]$ such that $\varphi(y)=1$ for $y \in [1,2]$. Then by the Mellin inversion formula (see, for example, \cite[Lemma 3.1]{Pra57}),
$\frac{1}{2\pi i} \int_{(1+\ep)} (\cM{\varphi})(s)y^{-s}ds=\varphi(y)$
for any $\ep>0$ and all $y>0$.
Then by positivity of $\lambda_{f \times \overline{f}}(\fr)$ and upon writing a dyadic sum in which $S$ takes values $2^j$ for $0 \le j \le \lfloor \log_2 R \rfloor$,
\begin{equation}\label{pole}
\sum_{\substack{\fr \in R(f) \\ \rN\fr \le R}} \frac{\lambda_{f \times \overline{f}}(\fr)}{\rN\fr}  \le  \sum_{S}\frac{1}{S} \sum_{\fr} \lambda_{f \times \overline{f}}(\fr) \varphi(\frac{\rN\fr}{S}) = \sum_{S}\frac{1}{S} \frac{1}{2\pi i} \int_{(1+\ep)} (\cM{\varphi})(s) S^s \left[ \sum_{\fr} \lambda_{f \times \overline{f}}(\fr) \rN\fr^{-s} \right] ds.
\end{equation}
Note that the inner sum over $\fr$ converges absolutely when $\Re(s)=1+\ep$. We truncate the integral to height $|t|=T'$, where $T'$ is a sufficiently large number (relative to $f,R,\varphi,\ep$). Then we move the line of integration to $\Re(s)=-U, |\Im(s)| \le T'$, where $U$ is also a sufficiently large number (relative to $T'$), picking up a simple pole at $s=1$. We choose $T',U$ sufficiently large such that the bound (\ref{tuchoice}) below holds.
Let $I(s)$ denote the integrand in (\ref{pole}) and recall the definition of the residue $s(f)$ in (\ref{sf}). The conclusion is
\begin{equation*}
 \frac{1}{2\pi i} \int_{(1+\ep)} I(s) ds = (\cM \varphi(1))Ss(f)+E_3
 \end{equation*}
 where
\begin{equation*}
 E_3 = \frac{1}{2\pi i} \left( \int_{-U-iT'}^{-U+iT'}+\int_{-U+iT'}^{1+\ep+iT'}-\int_{-U-iT'}^{1+\ep-iT'}+\int_{1+\ep-i\infty}^{1+\ep-iT'}+\int_{1+\ep+iT'}^{1+\ep+i\infty} \right) I(s) ds.
\end{equation*}

To estimate $\displaystyle{\sum_{\fr} \lambda_{f \times \overline{f}}(\fr) \rN\fr^{-s}=L^{\mathrm{ur}}(f \times \overline{f},s)}$ on $\Re(s)<0$, we first notice that
\begin{eqnarray*}
 |L^{\mathrm{ur}}(f \times \overline{f},s)| &=&\left| \frac{L(f\times \overline{f},s)}{L^{\mathrm{ra}}(f \times \overline{f},s)} \right|=|L(f\times \overline{f},s)|\prod_{\fp | \fq_{f \times \overline{f}}} \prod_{i=1}^{n} \prod_{j=1}^{n} \left|1-\alpha_i(\fp)\overline{\alpha_j(\fp)} \rN\fp^{-s}\right| \\
 & \ll & |L(f\times \overline{f},s)|\rN\fq_{f \times \overline{f}}^{-n^2\Re(s)}.
 \end{eqnarray*}
 By the functional equation (\ref{fe}), for $\Re(s)<0$,
 \[ |L(f\times \overline{f},s)|=(D_k^{n^2}\rN\fq_{f \times \overline{f}})^{\frac12-\Re(s)}\prod_{v \in S_\infty} \left| \frac{L(f_v \times \overline{f}_v,1-s)}{L(f_v \times \overline{f}_v,s)} \right| |L(f \times \overline{f},1-s)|. \]
 Using the equation (5.115) in \cite{IK04}, we have
 \[  \prod_{v \in S_\infty} \left| \frac{L(f_v \times \overline{f}_v,1-s)}{L(f_v \times \overline{f}_v,s)} \right| \ll  |\Im(s)|^{n^2(\frac12-\Re(s))}. \]
 Therefore,
 \[ |L^{\mathrm{ur}}(f \times \overline{f},s)| \ll (D_k^{n^2}\rN\fq_{f \times \overline{f}})^{\frac12-\Re(s)}  |\Im(s)|^{n^2(\frac12-\Re(s))}  \rN\fq_{f \times \overline{f}}^{-n^2\Re(s)}. \]
For $\Re(s)=-U, |\Im(s)| \le T'$, and for $|\Im(s)|=T', -U \le \Re(s) \le 1+\ep$, we know from (\ref{ms}) that $\cM{\varphi}(s)$ exhibits rapid decay. Thus, we can choose $T',U$ such that
\begin{equation}\label{tuchoice}
E_3=\frac{1}{2\pi i} \left( \int_{-U-iT'}^{-U+iT'}+\int_{-U+iT'}^{1+\ep+iT'}-\int_{-U-iT'}^{1+\ep-iT'}+\int_{1+\ep-i\infty}^{1+\ep-iT'}+\int_{1+\ep+iT'}^{1+\ep+i\infty} \right) I(s) ds=o(Ss(f)).
\end{equation}
Since $E_3=o(Ss(f))$, the right-hand side of (\ref{pole}) is now expressed as \[ \sum_{S}\frac{1}{S} \frac{1}{2\pi i} \int_{(1+\ep)} I(s) ds=\sum_{S}\frac{1}{S} (\cM \varphi(1))Ss(f)+E_3=\sum_{S} (s(f)+O(s(f)))\ll s(f)\log R. \] This finishes the proof of (\ref{112}).

For the lower bound (\ref{111}), since any integral ideal $\fr$ is in the support of $\varphi(\frac{\rN\fr}{S})$ for at most 3 dyadic values of $S$, we have
\begin{equation}\label{bfintjs}
\sum_{\substack{\fr \in R(f) \\ \rN\fr \le R}} \frac{\lambda_{f \times \overline{f}}(\fr)}{\rN\fr}  \ge  \frac13 \sum_{S}\frac{1}{2S} \sum_{\fr \in R(f)} \lambda_{f \times \overline{f}}(\fr) \varphi(\frac{\rN\fr}{S}) 
\end{equation}
where the sum over $S$ takes values $2^j$ for $\lfloor (1-\ep^{'})\log_2 R \rfloor \le j \le \lfloor \log_2 R \rfloor-1$ and $\ep^{'}>0$ is sufficiently small. The importance of our choice of the sum will be shown in (\ref{whyepprime}).
By Mellin inversion, the right-hand side of (\ref{bfintjs}) is equal to
\begin{equation}\label{intjs}
\sum_{S}\frac{1}{6S} \frac{1}{2\pi i} \int_{(1+\ep)} (\cM{\varphi})(s) S^s \left[ \sum_{\fr \in R(f)} \lambda_{f \times \overline{f}}(\fr) \rN\fr^{-s} \right] ds
\end{equation}
where $\ep>0$ is sufficiently small.

We truncate the integral to height $|t|=T^{''}$, a sufficiently large number (relative to $f,R,\varphi,\ep$). We choose $T^{''}$ such that the bound (\ref{tppchoice}) holds. Then we move the line of integration to $\Re(s)=1-2\delta+\ep, |\Im(s)| \le T^{''}$, where $\delta$ is as in (\ref{rfdef}). 
We denote the integrand in (\ref{intjs}) by $J(s)$ and have
\begin{equation*}
 \frac{1}{2\pi i} \int_{(1+\ep)} J(s) ds = (\cM\varphi(1))S\mathrm{Res}_{s=1}\left[ \sum_{\fr \in R(f)} \lambda_{f \times \overline{f}}(\fr) \rN\fr^{-s} \right]+E_4+E_5
 \end{equation*}
 where
 \begin{equation*}
 E_4=\frac{1}{2\pi i} \int_{1-2\delta+\ep-iT^{''}}^{1-2\delta+\ep+iT^{''}} J(s) ds
 \end{equation*}
 and
\begin{equation*}
E_5 = \frac{1}{2\pi i} \left( \int_{1-2\delta+\ep+iT^{''}}^{1+\ep+iT^{''}} - \int_{1-2\delta+\ep-iT^{''}}^{1+\ep-iT^{''}} + \int_{1+\ep-i\infty}^{1+\ep-iT^{''}} + \int_{1+\ep+iT^{''}}^{1+\ep+i\infty} \right) J(s) ds.
 \end{equation*}
Since $\cM\varphi(s)$ exhibits rapid decay, we can choose $T^{''}$ sufficiently large such that
\begin{equation}\label{tppchoice}
E_5=o(Ss(f)).
\end{equation}
Thus to verify (\ref{111}), it suffices to prove
\begin{equation}\label{suf1}
(\cM\varphi(1))S\mathrm{Res}_{s=1}\left[ \sum_{\fr \in R(f)} \lambda_{f \times \overline{f}}(\fr) \rN\fr^{-s} \right] \gg Ss(f)
\end{equation}
and 
$E_4=o(S).$

We define $\displaystyle{L^\ast(f \times \overline{f},s)=\sum_{\fr \in R(f)} \lambda_{f \times \overline{f}}(\fr) \rN\fr^{-s}}$ and consider the quotient $\displaystyle{\frac{L^{\mathrm{ur}}(f \times \overline{f},s)}{L^\ast(f \times \overline{f},s)}}$. Observe that for $\Re(s) \ge 1-2\delta+\ep$,
\begin{eqnarray*}
1 \ll \left| \frac{L^{\mathrm{ur}}(f \times \overline{f},s)}{L^\ast(f \times \overline{f},s)} \right|&=&\left| \prod_{\fp|\fP} (1-\lambda_{f \times \overline{f}}(\fp)\rN\fp^{-s})^{-1} \prod_{\fp: |\lambda_f(\fp)| \le \rN\fp^{-\delta}} (1-\lambda_{f \times \overline{f}}(\fp)\rN\fp^{-s})^{-1} \right| \\
& \asymp & \left| \prod_{\fp: |\lambda_f(\fp)| \le \rN\fp^{-\delta}} (1-\lambda_{f \times \overline{f}}(\fp)\rN\fp^{-s})^{-1} \right| \\
& \ll & \left| \prod_{\fp: |\lambda_f(\fp)| \le \rN\fp^{-\delta}} (1-\rN\fp^{-2\delta-s})^{-1} \right| \ll 1. 
\end{eqnarray*}
Hence, the quotient $\frac{L^{\mathrm{ur}}(f \times \overline{f},s)}{L^\ast(f \times \overline{f},s)}$ is absolutely convergent on $\Re(s) \ge 1-2\delta+\ep$, and in this region
\begin{equation}\label{lastlur}
L^\ast(f \times \overline{f},s) \asymp L^{\mathrm{ur}}(f \times \overline{f},s).
\end{equation}
Taking the residue at $s=1$, we know that (\ref{suf1}) holds.

By (\ref{lastlur}) and (\ref{ldl2}), we can estimate $E_4$ by
\begin{equation*} 
E_4
 \ll  \int_{1-2\delta+\ep-iT'}^{1-2\delta+\ep+iT'} |(\cM{\varphi})(s) S^s| |L^{\mathrm{ur}}(f \times \overline{f},s)| ds  \ll  \int_{1-2\delta+\ep-iT'}^{1-2\delta+\ep+iT'} |(\cM{\varphi})(s) S^s| |L(f \times \overline{f},s)| ds.
\end{equation*}
By the convexity bound (\ref{102}) and (\ref{bh}) in Lemma \ref{10}, rapid decay of $\cM\varphi(s)$, and the facts $S \ge R^{1-\ep^{'}}, R>q^C$, this is
\begin{eqnarray}\label{whyepprime}
\nonumber & \le & \int_{1-2\delta+\ep-iT'}^{1-2\delta+\ep+iT'} |(\cM{\varphi})(s) S^s|  \left[ \mathrm{Cond}(f \times \overline{f})(|\Im(s)|+2)^{n^2n_k} \right]^{\delta} ds \\
\nonumber & \le & \int_{1-2\delta+\ep-iT'}^{1-2\delta+\ep+iT'} |(\cM{\varphi})(s) S^s|  \left[ (\mathrm{Cond}(f))^{2n}(|\Im(s)|+2)^{n^2n_k} \right]^{\delta} ds \\
& \ll & S^{1-2\delta+\ep}q^{2nA\delta} \ll SR^{(-2\delta+\ep)(1-\ep^{'})+\frac{2nA\delta}{C}}=o(S).
\end{eqnarray}
The last equality holds if we choose $\ep>0,\ep^{'}>0$ sufficiently close to 0.
Therefore, $E_4=o(S)$.
This finishes the proof of (\ref{111}).
\end{proof}

\begin{proof}[Proof of Lemma \ref{9}]
We only prove the statements about $L^\flat(f,s)$ and $L^\sharp(f,s)$. The statements about $L^\flat(f \times g,s)$ and $L^\sharp(f \times g,s)$ follow analogously.

We define $L^\flat(f,s)$ as in (\ref{2.5}). Note that
\begin{equation*}
L^\sharp(f,s)=\frac{L^{\mathrm{ur}}(f,s)}{L^\flat(f,s)}=\Pi_1(s)\Pi_2(s)
\end{equation*}
where 
\begin{equation*}
\Pi_1(s)=\left( \prod_{\rN\fp<z} \prod_{j=1}^{n}(1-\alpha_j(\fp)\rN\fp^{-s})^{-1} \right), \ \ \Pi_2(s)= \prod_{\rN\fp \ge z} \frac{\prod_{j=1}^{n} (1-\alpha_j(\fp)\rN\fp^{-s})^{-1}}{1+\lambda_f(\fp)\rN\fp^{-s}}.
\end{equation*}
Since we assumed the Ramanujan-Petersson conjecture, we have $|\alpha_j(\fp)| \le 1$ so that each factor in the finite product $\Pi_1(s)$ is nonzero for any $\Re(s)>0$. Therefore, $\Pi_1(s)$ is holomorphic and has neither zero nor pole in $\Re(s)>0$, and is absolutely convergent for $\Re(s)>\ep$ for any $\ep>0$. For $\Pi_2(s)$, since 
$\lambda_f(\fp)=\sum_{j=1}^{n} \alpha_j(\fp)$ for every $p$, we have
\begin{equation*}
\Pi_2(s) = \prod_{\rN\fp \ge z} \frac{\prod_{j=1}^{n} (1-\alpha_j(\fp)\rN\fp^{-s})^{-1}}{(1-\lambda_f(\fp)\rN\fp^{-s})^{-1}} \cdot \frac{(1-\lambda_f(\fp)\rN\fp^{-s})^{-1}}{1+\lambda_f(\fp)\rN\fp^{-s}} =\prod_{\rN\fp \ge z} (1+O(\rN\fp^{-2s})) (1-\lambda_f^2(\fp)\rN\fp^{-2s})^{-1}.
\end{equation*}
Since we assumed the Ramanujan-Petersson conjecture, this is holomorphic and has neither zero nor pole in $\Re(s)>\frac12$, and is absolutely convergent for $\Re(s)>\frac12+\ep$ for any $\ep>0$. 
This finishes the proof of Lemma \ref{9}.
\end{proof}

\begin{proof}[Proof of Lemma \ref{12}]
We prove (\ref{3.7}). Fix $f,g$ and $\fr,\ft$ as in the Lemma. For squarefree $\fn$, (\ref{hd}) shows that
\begin{equation*}
\sum_{\fd|\fn} h(\fd)=\left( \prod_{\fp|(\fn,\fr)} \psi_f(\fp) \right) \left( \prod_{\fp|(\fn,\ft)} \psi_g(\fp) \right)=\psi_f((\fn,\fr))\psi_g((\fn,\ft)).
\end{equation*}
Moreover, recall the definition from (\ref{psidef}). We have
\[ \psi_{f,\fr}(\fn)\psi_{g,\ft}(\fn)=\mu_k(\fn)^2 \psi_f((\fn,\fr))\psi_g((\fn,\ft)).\]
Therefore, (\ref{3.7}) holds for all $\fn$.

We next prove (\ref{3.8}). 
If $g=\overline{f}$, then $\psi_g=\psi_f$ and for any $\fr \in R(f)$, $\ft \in R(\overline{f})$, squarefree $\fd|\fr\ft$,
\begin{equation*}
h(\fd)=\prod_{\substack{\fp|\fd \\ \fp \nmid (\fr,\ft)}} (\psi_f(\fp)-1) \prod_{\substack{\fp|\fd \\ \fp|(\fr,\ft)}} (\psi_f^2(\fp)-1).
\end{equation*}
Hence, by the definition of $\rho_f(\fd)$, and of $\psi_f(\fp)$,
\begin{eqnarray*}
h(\fd)\rho_f(\fd)|\lambda_f(\fd)|^2\rN\fd^{-1} &=& \prod_{\substack{\fp|\fd \\ \fp \nmid (\fr,\ft)}}(1+|\lambda_f(\fp)|^2\rN\fp^{-1})^{-1}|\lambda_f(\fp)|^2\rN\fp^{-1} (-\rN\fp|\lambda_f(\fp)|^{-2}-1) \\
& & \times \prod_{\substack{\fp|\fd \\ \fp|(\fr,\ft)}}(1+|\lambda_f(\fp)|^2\rN\fp^{-1})^{-1}|\lambda_f(\fp)|^2\rN\fp^{-1} (\rN\fp^2|\lambda_f(\fp)|^{-4}-1) \\
&=& \left( \prod_{\substack{\fp|\fd \\ \fp \nmid (\fr,\ft)}} (-1) \right) \left( \prod_{\substack{\fp|\fd \\ \fp|(\fr,\ft)}} (-\psi_f(\fp)-1) \right).
\end{eqnarray*}
Recalling that $h(\fd)$ is supported on squarefree divisors of $\fr\ft$, we have
\begin{equation*}
\sum_{\fd} h(\fd)\rho_f(\fd)|\lambda_f(\fd)|^2\rN\fd^{-1}
= \begin{cases}
0, & \text{ if } \fr \neq \ft \\
\prod_{\fp|\fr} (-\psi_f(\fp)), & \text{ if } \fr=\ft
\end{cases} 
\end{equation*}
which is equal to $\delta(\fr,\ft)|\psi_f(\fr)|$, as claimed.
\end{proof}
\end{subsection}

\end{section}

\begin{section}{Proof of Theorem \ref{main}}\label{pfofmainthm}

Now we proceed to prove Theorem \ref{main}, deducing it from the large sieve in Theorem \ref{14}.

\begin{subsection}{Reduction to well-spaced zeros}\label{redu}

To prove Theorem \ref{main}, we reduce our consideration to so-called well-spaced zeros. 
We partition the region $M(\alpha,T)$ into rectangles
\begin{equation}\label{rec}
R_j=[\alpha,1] \times [j\eta, (j+1)\eta],
\end{equation}
where $\eta$ is defined in (\ref{realeta}) and $j$ runs over all integers such that $-[\eta^{-1}T] \le j \le [\eta^{-1}T]$. Fix $q \ge 1$ and $S(q)$ as in Theorem \ref{main}. For each $f \in S(q)$, we arbitrarily choose one zero (if any) of $L(f,s)$ in $R_j$ for each $j$. Then we have a collection of zeros associated to each $f \in S(q)$. This collection is naturally divided into two subcollections -- one consisting of those zeros from all odd $j$ and another consisting of those from all even $j$. For each $f$, we define $Z(f)$ to be one of these subcollections with at least half of the chosen zeros. Note that for each $f$, the zeros in $Z(f)$ are $\eta$-well-spaced. 

The following lemma shows an upper bound for the number of zeros in each $R_j$.
\begin{lem}\label{rj}
Fix $\frac34 \le \alpha<1$ and $\eta$ as above. For a fixed $f \in S(q)$, each region $R_j$ can contain at most
\begin{equation}\label{pra}
 \ll (1-\alpha) \log (q(|\frac{2j+1}{2}\eta|+3))+1
 \end{equation}
zeros of $L(f,s)$. 
\end{lem}
Thus we can conclude that
\begin{equation*}
\sum_{f \in S(q)} N(f;\alpha,T) \ll 2((1-\alpha) \log (qT)+1) \sum_{f \in S(q)} |Z(f)|,
\end{equation*}
and we may restrict our attention to well-spaced zeros.
Since
$(1-\alpha) \log (qT)+1 \le e^{(1-\alpha) \log (qT)}=(qT)^{1-\alpha}$,
the proof of Theorem \ref{main} is reduced to proving
(\ref{main2})
for all $q \ge 1, T \ge 2$. Now we prove Lemma \ref{rj}.
The proof is similar with that of Lemma 2.1 in \cite[pp. 331-332]{Pra57}.

We recall the known zero-free region for automorphic $L$-functions; see \cite[Theorem 5.42]{IK04}.
\begin{lemma}\label{dzero}
Let $f \in (S(q))_{q \ge 1}$. Then there exists a constant $c>0$ depending only on $n,n_k$ such that $L(f,s)$ has no zeros in the region
\begin{equation}\label{zerofree}
\{s=\sigma+it \in \C: \sigma \ge 1-\frac{c}{\log (\mathrm{Cond}(f)(|t|+3))} \}
\end{equation}
except possibly one simple real zero $\beta_f<1$.
\end{lemma}

We let 
\begin{equation}\label{cvalue}
C=\frac{c}{6}.
\end{equation}
Fix $q,T$. Choose $\eta=\frac{C}{\log qT}$ as in (\ref{realeta}). 
We denote $t_j=\frac{2j+1}{2}\eta$ and $r=1-\alpha$. 

If $r<\frac{c}{2\log (\mathrm{Cond}(f)(|t_j|+3))}$, we know that $R_j$ lies completely in the standard zero free region (\ref{zerofree}) and there is at most one zero in the region. Therefore, each $R_j$ contains at most one zero of $L(f,s)$ and Lemma \ref{rj} holds in this case.

Now we assume that $r \ge \frac{c}{2\log (\mathrm{Cond}(f)(|t_j|+3))}$. Let $G(t_j,r)$ be the circular disc
\begin{equation}\label{62.3}
\{s: |s-(1+it_j)| \le r\}
\end{equation}
and let $Q(t_j,r)$ denote the number of zeros of $L(f,s)$ in the interior of the circle $G(t_j,r)$. 
Since $r \ge \frac{c}{2\log (\mathrm{Cond}(f)(|t_j|+3))} \ge \frac{c}{3\log (\mathrm{Cond}(f) T)}=\frac{2C}{\log (\mathrm{Cond}(f) T)}$, $G(t_j,r)$ contains the region $R_j$ and it suffices to show that
\begin{equation}\label{62.5}
Q(t_j,r) \ll r\log \mathrm{Cond}(f) (|t_j|+3).
\end{equation}

Since $\alpha \ge \frac34$, we have $r \le \frac14$. From (5.28) of \cite[Chapter 5]{IK04}, we have for $-\frac12 \le \Re(s) \le 2$,
\begin{equation}\label{62.6}
\frac{L'}{L}(f,s)+\frac{1}{s}+\frac{1}{s-1}- \sum_{\substack{v \in S_\infty \\ |s+\mu_{f}(v,i)|<1}} \frac{1}{s+\mu_{f}(v,i)}-\sum_{|s-\rho|<1} \frac{1}{s-\rho} \ll \log (\rN \fq_f \prod_{v \in S_\infty } \prod_{i=1}^n  (1+|it+\mu_{f}(v,i)|^{d(v)})),
\end{equation}
where $\mu_{f}(v,i)$ and $d(v)$ are parameters in the analytic conductor of $f$; see (\ref{condf}).
We put $s=1+r+it_j$ (so $\sigma=\Re(s)=1+r$) and consider the real part of this formula. Since $\Re(\mu_{f}(v,i)) >-1$ (see p.94 of \cite{IK04}) and the Ramanujan-Petersson conjecture is assumed, we have
\begin{equation}\label{62.7}
\Re \frac{L'}{L}(f,s) \le \left| \frac{L'}{L}(f,s) \right| \le \sum_\fm \frac{|\Lambda_f(\fm)|}{\rN\fm^\sigma} \le -n\frac{\zeta_k^{'}}{\zeta_k}(\sigma) \ll \frac{1}{\sigma-1}=\frac{1}{r}
\end{equation}
where $\Lambda_f(\fn)$ is the von Mangoldt function associated to $f$, which is supported on prime powers, for which it satisfies $\Lambda_f(\fp^k)=\sum_{i=1}^{n} \alpha_i(\fp)^k \log \rN\fp$ (see \cite[(5.26)]{IK04}). The fact $\Re(\mu_{f}(v,i)) >-1$ also gives the bound 
\[ \Re \sum_{\substack{v \in S_\infty \\ |s+\mu_{f}(v,i)|<1}} \frac{1}{s+\mu_{f}(v,i)} \le \frac{nn_k}{r} \ll \frac{1}{r}. \]

For $s=1+r+it_j$ we always have $\Re(1/(s-\rho))>0$ for any zero $\rho$ of $L(f,s)$. For $r \le \frac14$, because every $\rho=\beta+i\gamma \in G(t_j,r)$ is also in $|s-\rho| \le 1$ and because of
\begin{equation}\label{62.9}
\Re \left( \frac{1}{s-\rho} \right)=\frac{\Re(s-\rho)}{|s-\rho|^2} \ge \frac{r}{(2r)^2}=\frac{1}{4r},
\end{equation}
we have
\begin{equation}\label{62.8}
\Re \left( \sum_{|s-\rho| < 1} \frac{1}{s-\rho} \right) \ge \Re \left( \sum_{\rho \in G(t_j,r)} \frac{1}{s-\rho} \right) \ge \frac{Q(t_j,r)}{4r}.
\end{equation}
We put this into (\ref{62.6}) with $s=1+r+it_j$ and it follows that
\[ \frac{Q(t_j,r)}{4r} \ll \frac{1}{r}+O(\log \mathrm{Cond}(f)(|t_j|+3)). \]
For $r \ge \frac{c}{2\log (\mathrm{Cond}(f)(|t_j|+3))}$ the right side is $\ll \log \mathrm{Cond}(f)(|t_j|+3)$ and (\ref{62.5}) holds. 
This finishes the proof of Lemma \ref{rj}. 
\end{subsection}

\begin{subsection}{Key lemmas for the zero detector $z_\fr(f,s)$}
The proof of Theorem \ref{main} has now been reduced to proving (\ref{main2}) for sets $Z(f)$ of $\eta$-well-spaces zeros of $L(f,s)$. We now focus on developing the precise definition of the zero detector $z_\fr(f,s)$ and proving its key properties, leading to the proof of Lemma \ref{181}.

Let $1 \le w<y$ be two parameters to be specified later (see (\ref{par1}) and (\ref{par2})). We use the Selberg weights
\begin{equation*}
\lambda_\fd:=\mu_k(\fd)m(\rN\fd)
\end{equation*}
where for a positive integer $d$,
\begin{equation}\label{defmd}
m(d):=
\begin{cases}
1, & \text{ if } d \le w, \\
\frac{\log(y/d)}{\log(y/w)}, & \text{ if } w \le d \le y, \\
0, & \text{ if } y<d.
\end{cases}
\end{equation}
Then we define
\begin{equation*}
\Delta(\fn):=\sum_{\fd|\fn} \lambda_\fd.
\end{equation*}
Note that $\Delta(\fn)=0$ for $1<\rN\fn \le w$.

Fix $q \ge 1$ and $f \in S(q)$. Fix $\alpha,T$ and recall the region $M(\alpha,T)$ defined in (\ref{mat}).
We recall that we shall use $z_\fr(f,s)$ defined in (\ref{zr}) as our zero detector for zeros of $L(f,s)$ inside the region $M(\alpha,T)$. For a real number $x>y$ to be chosen in (\ref{par2}), let 
\begin{equation*}
z_\fr(f,s):=\sideset{}{_{}^{\flat}}\sum_{\substack{w \le \rN\fn \le x \\ (\fn,\fP)=1}} \Delta(\fn) \psi_{f,\fr}(\fn)e^{-\rN\fn(\log qT)^2/x}\lambda_f(\fn)\rN\fn^{-s},
\end{equation*}
i.e, this now makes the expression (\ref{zr}) precise (up to the choices of $w,y,x$), with the choice
\begin{equation}\label{par0}
a_\fn=\Delta(\fn)e^{-\rN\fn(\log qT)^2/x}.
\end{equation}
Recall that the notation $\sum^\flat$ denotes a sum over squarefree ideals. The following proposition explains the reason for the name ``zero detector''.

\begin{prop}\label{18}
Fix $q \ge 1$. Let $f \in S(q)$, $\fr \in R(f)$ with $\rN\fr \le R$, and fix $T \ge 1$, $\frac12 \le \alpha < 1$. If $\rho \in M(\alpha,T)$ is a zero of $L(f,s)$, then
\begin{equation}\label{zdbound}
1 \ll_{\ep}z_\fr(f,\rho),
\end{equation}
provided that
\begin{eqnarray}\label{4.4}
& & x \ge (\log qT)^2(yq^{A/2}T^{nn_k/2}R^{1+4\delta})^{1/(2\alpha-1)+\ep} \\
\nonumber & & \log (yR) \ll \log (qT).
\end{eqnarray}
\end{prop}

\begin{rem}\label{rema}
Lemma \ref{181} in Section \ref{secofls} follows after applying Lemma \ref{11} (under the assumption $R>q^C$), and then Proposition \ref{18}. In particular,
\begin{eqnarray*}
\nonumber \sum_{f \in S(q)}|Z(f)| &=& \frac{1}{\log R} \sum_{f \in S(q)} \frac{|Z(f)|}{s(f)}s(f)\log R \ll \frac{1}{\log R} \sum_{f \in S(q)} \frac{|Z(f)|}{s(f)}\sum_{\substack{\fr \in R(f) \\ \rN\fr \le R}} \frac{1}{|\psi_f(\fr)|}  \\
& \ll & \frac{1}{\log R} \sum_{f \in S(q)} \frac{1}{s(f)} \sum_{\rho \in Z(f)} \sum_{\substack{\fr \in R(f) \\ \rN\fr \le R}} \frac{1}{|\psi_f(\fr)|} \left| z_\fr(f,\rho) \right|^2.
\end{eqnarray*}
\end{rem}

For the proof of Lemma \ref{201} in Section \ref{secofls}, we need the following lemma, which is the analog of a Graham's lemma. We prove the lemma in \cite{An22}.
\begin{lem}[Corollary 1.2 in \cite{An22}]\label{15}
For any $\alpha$ with $1/2<\alpha<1$, for all $x \ge 1$,
\begin{equation*}
\sideset{}{_{}^{\flat}}\sum_{\rN\fn \le x} \Delta(\fn)^2\rN\fn^{1-2\alpha} \ll \frac{\log(x/w)}{\log(y/w)}x^{2-2\alpha}.
\end{equation*}
\end{lem}

\begin{rem}\label{remb}
Lemma \ref{201} will follow immediately by Lemma \ref{15}, provided that
\begin{equation}\label{2.38}
\log x \ll \log(qT), \ \ \ \log(\frac{x}{w}) \ll \log(\frac{y}{w}).
\end{equation}
In particular, under these assumptions, the left-hand side of (\ref{201eqn}) is bounded by
\[ \ll \log(qT)\sideset{}{_{}^{\flat}}\sum_{\rN\fn \le x} |\Delta(\fn)|^2e^{-2\rN\fn(\log qT)^2/x}\rN\fn^{1-2\alpha} \ll (\log qT) \sideset{}{_{}^{\flat}}\sum_{\rN\fn \le x} |\Delta(\fn)|^2\rN\fn^{1-2\alpha} \ll (\log qT)x^{2(1-\alpha)}. \]
\end{rem}

\end{subsection}

\begin{subsection}{Proof of Proposition \ref{18}}
Fix the parameters as stated in Proposition \ref{18}. For simplicity of notation, denote $X=x(\log qT)^{-2}$. By Mellin inversion (see e.g., \cite[Lemma 3.2]{Pra57}), for $y>0,b>0$, we have
\begin{equation}\label{pra32}
e^{-y}=\int_{(b)} \Gamma(s)y^{-s}ds.
\end{equation}
Let $\rho \in M(\alpha,T)$ be a zero of $L(f,s)$. Using (\ref{pra32}), we have
\begin{eqnarray}\label{pra32appl}
\nonumber & & e^{-1/X}+z_\fr(f,\rho)+\sideset{}{_{}^{\flat}}\sum_{\substack{(\fn,\fP)=1 \\ \rN\fn > x}} \Delta(\fn)\psi_{f,\fr}(\fn)e^{-\rN\fn/X}\lambda_f(\fn)\rN\fn^{-\rho}=\sideset{}{_{}^{\flat}}\sum_{\substack{(\fn,\fP)=1 }} \Delta(\fn)\psi_{f,\fr}(\fn)e^{-\rN\fn/X}\lambda_f(\fn)\rN\fn^{-\rho} \\
&=& \frac{1}{2\pi i}\int_{(3)} \sideset{}{_{}^{\flat}}\sum_{\substack{(\fn,\fP)=1}} \Delta(\fn)\psi_{f,\fr}(\fn)\lambda_f(\fn)\rN\fn^{-(s+\rho)} \Gamma(s)X^s ds.
\end{eqnarray}
We will show in (\ref{rhso1}) and in (\ref{ngreaterx}) that the right-hand side of (\ref{pra32appl}) and the sum over $\rN\fn>x$ are both $o(1)$ (if we choose $qT$, hence $X$, large enough). Since $e^{-1/X} \gg 1$, we can conclude Proposition \ref{18} holds.

We directly compute
\begin{eqnarray*}
& &\sideset{}{_{}^{\flat}}\sum_{\substack{(\fn,\fP)=1 }} \Delta(\fn)\psi_{f,\fr}(\fn)\lambda_f(\fn)\rN\fn^{-s} \\
&=& \sideset{}{_{}^{\flat}}\sum_{\substack{(\fd,\fP)=1 }} \lambda_\fd \psi_{f,\fr}(\fd)\lambda_f(\fd)\rN\fd^{-s}\sideset{}{_{}^{\flat}}\sum_{\substack{(\fn,\fd\fP)=1 }} \psi_{f,\fr}(\fn)\lambda_f(\fn)\rN\fn^{-s} \\
&=& \sideset{}{_{}^{\flat}}\sum_{\substack{(\fd,\fP)=1 }} \lambda_\fd \psi_{f,\fr}(\fd)\lambda_f(\fd)\rN\fd^{-s} \prod_{(\fp,\fd\fP)=1}(1+\psi_{f,\fr}(\fp)\lambda_f(\fp)\rN\fp^{-s}) \\
&=& \sideset{}{_{}^{\flat}}\sum_{\substack{(\fd,\fP)=1 }} \lambda_\fd \psi_{f,\fr}(\fd)\lambda_f(\fd)\rN\fd^{-s} \prod_{\fp|\fr,\ \fp \nmid \fd} (1+\psi_f(\fp)\lambda_f(\fp)\rN\fp^{-s}) \prod_{(\fp,\fd\fr\fP)=1}(1+\lambda_f(\fp)\rN\fp^{-s})  \\
&=& \sideset{}{_{}^{\flat}}\sum_{\substack{(\fd,\fP)=1 }} \lambda_\fd \psi_{f,\fr}(\fd)\lambda_f(\fd)\rN\fd^{-s}  \prod_{\fp|\frac{\fr}{(\fr,\fd)}}(1+\psi_f(\fp)\lambda_f(\fp)\rN\fp^{-s}) L^\flat(f,s) \prod_{\fp|\fr\fd}(1+\lambda_f(\fp)\rN\fp^{-s})^{-1} \\
&=& L^\flat(f,s)M_\fr(f,s)
\end{eqnarray*}
where we define the ``mollifier''
\begin{equation*}
M_\fr(f,s):=\sideset{}{_{}^{\flat}}\sum_{\substack{(\fd,\fP)=1}} \lambda_\fd \psi_{f,\fr}(\fd)\lambda_f(\fd)\rN\fd^{-s} \prod_{\fp|\frac{\fr}{(\fr,\fd)}}(1+\psi_f(\fp)\lambda_f(\fp)\rN\fp^{-s}) \prod_{\fp|\fr\fd}(1+\lambda_f(\fp)\rN\fp^{-s})^{-1}.
\end{equation*}
By (\ref{defmd}), the sum in this definition is supported only on squarefree ideals $\fd$ such that $1 \le \rN\fd \le y$, $(\fd,\fP)=1$.
Therefore, the right-hand side of (\ref{pra32appl}) is
\begin{equation*}
=\frac{1}{2\pi i}\int_{(3)} L^\flat(f,s+\rho)M_\fr(f,s+\rho) \Gamma(s)X^s ds.
\end{equation*}

Let $0<\ep^{''}<\alpha-\frac12$ be a sufficiently small number and let $\sigma:=\frac12+\ep^{''}-\alpha$.
We move the line of integration to $\Re(s)=\sigma$, and obtain
\begin{equation}\label{3tosigma}
\frac{1}{2\pi i}\int_{(3)}L^\flat(f,s+\rho)M_\fr(f,s+\rho)\Gamma(s)X^s ds=\frac{1}{2\pi i}\int_{(\sigma)}L^\flat(f,s+\rho)M_\fr(f,s+\rho)\Gamma(s)X^s ds
\end{equation}
because at $s=0$ the pole of $\Gamma(s)$ is cancelled by the zero of $L^\flat(f,s+\rho)$. Since $\sigma+\Re(\rho)>\sigma+\alpha=\frac12+\ep^{''}$, we have by the convexity bound (\ref{convf})
\begin{equation*}
|L^\flat(f,s+\rho)| \ll_\ep \left( q^A(|t+\Im(\rho)|+2)^{nn_k} \right)^{\frac14-\frac{\ep^{''}}{2}+\ep}
\end{equation*}
for $s=\sigma+it$; here we use the parameter $A$ associated to the family $(S(q))_{q \ge 1}$.

For the estimation of $M_\fr(f,s+\rho)$, we need the following lemma.

\begin{lem}\label{mollifierbd}
Let $f \in S(q)$, $\fr \in R(f)$. For any $z \ge 4n^2n_k^2$, any $s \in \C$ with $\frac12 \le \Re(s) \le 1$ and any $\ep>0$, we have
\begin{equation*}
M_\fr(f,s) \ll_\ep \rN\fr^{1+2\delta-\Re(s)+\ep} y^{1-\Re(s)+\ep}.
\end{equation*}
\end{lem}

\begin{proof}[Proof of Lemma \ref{mollifierbd}]
Here we denote the tau function $\tau(\fu)$ to be the number of prime ideal divisors of $\fu$; similar to the usual tau function, it has the property $\tau(\fu) \ll \rN\fu^{\ep}$ for any $\ep>0$.

For all squarefree ideals $\fd$ and such that $(\fd,\fP)=1$, for any $\fr \in R(f)$, and for any $\ep>0$, we have the bound
\begin{equation}\label{psilam}
|\psi_{f,\fr}(\fd)\lambda_f(\fd)| \le (\rN(\fr,\fd))^{1+\delta}\lambda_{f}(\frac{\fd}{(\fr,\fd)}) \ll_\ep (\rN(\fr,\fd))^{1+\delta}\rN\fd^{\frac{\ep}{2}} \le \rN\fr^\delta(\rN(\fr,\fd))\rN\fd^{\frac{\ep}{2}}.
\end{equation}
Since $\Re(s) \le 1$, we have
\begin{equation*}
\left| \prod_{\fp|\frac{\fr}{(\fr,\fd)}}(1+\psi_f(\fp)\lambda_f(\fp)\rN\fp^{-s}) \right| \le \tau(\fr)\rN\fr^{1+\delta-\Re(s)}.
\end{equation*}
Choose $z$ to be any constant such that 
\begin{equation}\label{zsecond}
z \ge 4n^2.
\end{equation}
Combining the previous condition (\ref{zfirst}) on $z$, we choose 
\begin{equation}\label{zvalue}
z=4n^4n_k^2.
\end{equation}
Since $|\lambda_f(\fp)| \le n$ for any prime ideal $\fp \in R(f)$ (so that in particular $\rN\fp>z$) and $\Re(s) \ge \frac12$, we have 
\begin{equation*}
|1+\lambda_f(\fp)\rN\fp^{-s}| \ge 1-n\rN\fp^{-\Re(s)} \ge 1-nz^{-\frac12} \ge \frac12
\end{equation*}
and hence for any $\ep>0$,
\begin{equation*}
\left| \prod_{\fp|\fr\fd}(1+\lambda_f(\fp)\rN\fp^{-s})^{-1} \right| \le \tau(\fr\fd) \ll_\ep (\rN\fr\rN\fd)^\frac{\ep}{2}.
\end{equation*}
Then we compute
\begin{equation}\label{mrfscomp}
M_\fr(f,s)  \ll_\ep   \sideset{}{_{}^{\flat}}\sum_{\substack{(\fd,\fP)=1 \\ 1 \le \rN\fd \le y}} \rN\fr^{1+2\delta-\Re(s)+\frac{\ep}{2}}\rN\fd^{\ep-\Re(s)}(\rN(\fr,\fd)) 
 =\rN\fr^{1+2\delta-\Re(s)+\frac{\ep}{2}}\sideset{}{_{}^{\flat}}\sum_\fu \rN\fu \sideset{}{_{}^{\flat}}\sum_{\fv} (\rN\fv\rN\fu)^{\ep-\Re(s)}
 \end{equation}
where $\fu=(\fr,\fd)$, $\fv=\frac{\fd}{(\fr,\fd)}$. Here the sum over $\fu$ takes all squarefree $\fu$ with $1 \le \rN\fu \le y, \fu|\fr, (\fu,\fP)=1$; the sum over $\fv$ takes all squarefree $\fv$ with $1 \le \rN\fv \le y, (\fv,\fu\fP)=1$. Since $\rN\fv=\frac{\rN\fd}{\rN\fu} \le \frac{y}{\rN\fu}$, the double sum in (\ref{mrfscomp}) is equal to
 \begin{equation*}
\sum_{\fu} \rN\fu^{1+\ep-\Re(s)} \sum_{\fv} \rN\fv^{\ep-\Re(s)} \ll  \sum_\fu \rN\fu^{1+\ep-\Re(s)} (\frac{y}{\rN\fu})^{1+\ep-\Re(s)} = \sum_{\fu} y^{1+\ep-\Re(s)} \ll y^{1+\ep-\Re(s)} \tau(\fr).
 \end{equation*}
Therefore, 
\begin{equation*}
M_\fr(f,s) \ll_\ep \rN\fr^{1+2\delta-\Re(s)+\frac{\ep}{2}} y^{1+\ep-\Re(s)} \tau(\fr) \ll_\ep  \rN\fr^{1+2\delta-\Re(s)+\ep} y^{1+\ep-\Re(s)}.
\end{equation*}
\end{proof}

By Lemma \ref{mollifierbd} and the fact that $\Gamma(s)$ has exponential decay for large $|\Im(s)|$, the integral in (\ref{3tosigma}) on $\Re(s)=\sigma$ is
\begin{equation}\label{rhso1}
\ll_\ep X^{1/2-\alpha+\ep}R^{1/2+2\delta+\ep}y^{1/2+\ep}q^{A/4+\ep}T^{nn_k/4+\ep}=o(1)
\end{equation}
by (\ref{4.4}). Therefore, the right-hand side of (\ref{pra32appl}) is $o(1)$.

Now we consider the terms in the left-hand side of (\ref{pra32appl}).
Using $|\Delta(\fn)| \le y$ for all $\fn$ and (\ref{psilam}), the tail contribution from the series summing over $\rN\fn>x$ is
\begin{equation}\label{ngreaterx}
\sideset{}{_{}^{\flat}}\sum_{\substack{(\fn,\fP)=1 \\ \rN\fn > x}} \Delta(\fn)\psi_{f,\fr}(\fn)e^{-\rN\fn/X}\lambda_f(\fn)\rN\fn^{-\rho}  \ll yR^{1+\delta}e^{-(\log(qT))^2}\sum_{\rN\fn > x} e^{-\frac{\rN\fn-x}{X}} \rN\fn^{-\frac12+\ep}=o(1).
\end{equation}
Here we used the fact that any $n$ has at most $\ll \frac{\log n}{\log\log n}$ prime divisors, thus $n$ corresponds to at most $\ll n_k^{\frac{\log n}{\log \log n}} \ll n^\ep$ ideals $\fn$ such that $\rN\fn=n$.
Thus, the value of the Dirichlet polynomial $z_\fr(f,s)$ at $s=\rho$ is $\gg_\ep 1$. The proof of Proposition \ref{18} is finished.

\end{subsection}

\begin{subsection}{Proof of Theorem \ref{main}}

Recall that after the reduction to $\eta$-well-spaced zeros with $\eta$ as in (\ref{realeta}), our goal is to prove (\ref{main2}).

By Lemma \ref{181}, Theorem \ref{14}, and Lemma \ref{201} (the statements of which are now made precise via Remarks \ref{rema} and \ref{remb}), we know that if $\log R \gg \log qT$, $R>q^{C}$ where $C>nA$ (see Lemma \ref{11}), and the assumptions (\ref{4.4}), (\ref{3.10}) (with $N=x$), (\ref{2.38}) are satisfied, then we have
\begin{align*}
\sum_{f \in S(q)} |Z(f)| & \ll  \frac{1}{\log R} \sum_{f \in S(q)} \frac{1}{s(f)} \sum_{\rho \in Z(f)} \sum_{\substack{\fr \in R(f) \\ \rN\fr \le R}} \frac{1}{|\psi_f(\fr)|} \left| z_\fr(f,\rho) \right|^2 && \text{(by Lemma \ref{181})} \\
& \ll  \frac{1}{\log R} \log(qTx)\left( 1+\log \frac{\log x}{\log qTR} \right)\sideset{}{_{}^{\flat}}\sum_{\rN\fn \le x} |a_\fn|^2\rN\fn^{1-2\alpha}  && \text{(by Theorem \ref{14})} \\
& \ll x^{2(1-\alpha)}. && \text{(by Lemma \ref{201})} 
\end{align*}
We choose the parameters as follows: for sufficiently small numbers $\ep_i>0$, $i=1,2,3,4$,
\begin{equation}\label{par1}
\delta=\ep_1, \ \ R=q^{nA}(qT)^{\ep_2}, \ \ w=
M=2\left( q^{d+nA/2}TR^{1+3\delta}(\log R) \right)^{\frac{1}{\frac{1}{2}-\ep_0}}
\end{equation}
\begin{equation}\label{par2}
y=w(qT)^{\ep_3}, \ \ x=\left[ yq^{A/2}T^{nn_k/2}R^{1+4\delta} \right]^{1/(2\alpha-1)+\ep_4}.
\end{equation}
These choices fulfill all the assumptions for $qT$ sufficiently large (depending only on $\ep_i$'s).
Setting $x$ as above, we have (\ref{main2})
for $qT \ge C_{\ep}$, where $C_{\ep}$ is a sufficiently large number depending only on $\ep_i$'s. For $qT<C_{\ep}$, we use the trivial bound for the number of zeros; see, e.g., \cite[Theorem 5.8]{IK04}, and obtain $\sum_{f \in S(q)} N(f;\alpha,T) \ll_{\ep} 1$. Now we have finished the proof of (\ref{main2}), and hence of Theorem \ref{main}.
\end{subsection}

\end{section}

\begin{section}{Application to an effective Chebotarev density theorem}\label{sectcdt}

Our log-free zero density estimate for automorphic $L$-functions leads to a new effective Chebotarev density theorem. Our approach is similar to that in \cite{PTW17}, and applies a refinement in \cite{TZ18}. We will work somewhat more generally before specializing to the context of Theorem \ref{cdt}.

In 1975, Lagarias and Odlyzko \cite{LO75} gave the first effective version of the Chebotarev density theorem: for any finite Galois extension of number fields $L/k$, they provided explicit error terms either assuming the Generalized Riemann Hypothesis or without the Generalized Riemann Hypothesis. Instead of considering the extension $L/k$, they considered the problem for the extension $L/E$ where $E$ is a number field lying between $L$ and $k$ and $L/E$ is cyclic. In this way, they were able to write the quotient of Dedekind zeta-functions $(\zeta_L)/(\zeta_E)$ into a product of Hecke $L$-functions, which are known to be analytic.

Researchers have been trying to improve the error term of the effective Chebotarev density theorem since Lagarias and Odlyzko's proof was published. Serre \cite{Ser81} improved the error term in the theorem assuming the Generalized Riemann Hypothesis by a log factor. For almost every field in many families of number fields, Pierce, Turnage-Butterbaugh, and Wood \cite{PTW17} eliminated the exceptional zero term and gave a better threshold so that the effective Chebotarev density theorem counts smaller primes unconditionally. 
Since then, other improvements have been shown in work including \cite{TZ18}, \cite{TZ19b}, \cite{An18}. For our purposes, we will adapt \cite{PTW17}, \cite{TZ18}, developing some new ideas required in our setting over an arbitrary base field $k$.

We recall an effective Chebotarev density theorem conditional on a zero-free region. 
\begin{thm}\label{cdtzf}
Let $k$ be a fixed number field. Fix $0<\delta \le \frac{1}{2}$, and an integer $n \ge 2$. Let $G$ be a fixed transitive subgroup of $S_n$. Assume the strong Artin conjecture for $G$ (see Conjecture \ref{saconj}). Then for any Galois extension of number fields $L/k$ with $\mathrm{Gal}(L/k) \simeq G$ and such that the Artin $L$-function $\zeta_L(s)/\zeta_k(s)$ is zero-free in the region
\begin{equation}\label{lkzf}
[1-\delta,1] \times [-T,T]
\end{equation}
where $T \ge (\log D_L)^{24}$, we have that for any conjugacy class $\cC \subseteq G$,
\begin{equation*}
\left| \pi_\cC(x,L/k)-\frac{|\cC|}{|G|}\pi(x) \right| \le \frac{|\cC|}{|G|}\frac{x}{\log x}\left( x^{-\delta/8}+T^{-\frac{1}{24}}e^{-\frac{1}{24}\sqrt{c_4(\log x)/n_L}}+T^{-\frac{1}{24}}e^{-\frac{1}{24}\frac{c_4 \log x}{\log D_L}}\right)
\end{equation*}
for an absolute constant $c_4>0$ and all
$x \ge (\log D_L)^{16/\delta}$.
\end{thm}

Theorem \ref{cdtzf} is analogous to \cite[Theorem 8.3]{TZ18} and we will omit the proof here. Although in \emph{loc. cit.} the theorem is stated for field extensions over $\Q$ for this particular proof, we can easily generalize it to an arbitrary base field; see Remark (2) after Proposition 8.1 in \cite{TZ18}.

For a number field $k$ and a transitive subgroup $G \subseteq S_n$ with $n \ge 2$ an integer, we denote
\[ Z_n(k,G)=\{ K/k: [K:k]=n, \mathrm{Gal}(\widetilde{K}/k) \simeq G \}. \]
Following the method of \cite{PTW17}, we need to impose a ramification restriction. For a set of conjugacy classes $\cF$ in $G$, we let
\begin{eqnarray}\label{fml2}
\nonumber Z_n^\cF(k,G)&=&\{K \in Z_n(k,G): \text{for the extension $\widetilde{K}/k$, the inertia group of every tamely} \\
& & \text{ramified prime ideal in $k$ is generated by an element in $\cF$}. \}
\end{eqnarray}
Let 
\begin{equation}\label{fml3}
Z_n^\cF(k,G;X)=\{K \in Z_n^\cF(k,G): \mathrm{Nm}_{k/\Q}\mathrm{Disc}(K/k) \le X \}.
\end{equation}
The family $Z_n^\ast(k;X)$ defined in Section \ref{intro} is exactly $Z_n^\cF(k,C_n;X)$.

With Theorem \ref{cdtzf} in hand, we will prove in Section \ref{pfgencdt} the following theorem that specializes to Theorem \ref{cdt} in the case $Z_n^\cF(k,C_n)$. It is a ``meta theorem'' with significant hypotheses that aims to demonstrate in general how our zero density estimate (Theorem \ref{main}) and Theorem \ref{cdtzf} can be used to obtain an effective Chebotarev density theorem for suitable families of number fields.
\begin{thm}\label{gencdt}
Let $Z_n^\cF(k,G)$ be a family of fields as defined in (\ref{fml2}). Assume the strong Artin conjecture for $G$.
Assume that there exists $\beta>0$ such that for any $X \ge 1$, $|Z_n^\cF(k,G;X)| \gg X^\beta$. Assume that there exists $0 \le \tau<\beta$ such that for any field $F$ in a certain set $\mathfrak{F}$ of number fields (defined in (\ref{mfrakf})), and any $\ep_1>0$, there are $\ll X^{\tau+\ep_1}$ fields $K$ in $Z_n^\cF(k,G;X)$ such that $\widetilde{K}$ is an extension of $F$.
Then for any $0<\ep<1$ sufficiently small, there exists $\kappa=\kappa(n,n_k,|G|,\ep)>0$ such that aside from at most $\ll_{\ep} X^{\tau+\ep}$ possible exceptions, each field $K \in Z_n^\cF(k,G;X)$ has the property that for every conjugacy class $\cC \subseteq G$,
\begin{equation*}
\left| \pi_\cC(x,\widetilde{K}/k)-\frac{|\cC|}{|G|}\pi(x) \right| \ll
\begin{cases}
\frac{|\cC|}{|G|}x^{1-\kappa} & \text{if } (\log D_{\widetilde{K}})^{2/\kappa} \le x<D_{\widetilde{K}}^{1/(24\kappa)}, \\
\frac{|\cC|}{|G|} \frac{x}{\exp(c_3(\log x)^{1/2}|G|^{-1/2}n_k^{-1/2})} & \text{if } x \ge D_{\widetilde{K}}^{1/(24\kappa)},
\end{cases}
\end{equation*}
where $c_3>0$ is an absolute constant. Since $\beta>\tau$, the set of possible exceptional fields in the family has density zero.
\end{thm}
The set of fields $\mathfrak{F}$ will be defined as a set of intermediate fields between $k$ and $\widetilde{K}$ with $K \in Z_n^\cF(k,G;X)$. Thus, we will refer to the corresponding assumption  in Theorem \ref{gencdt} as the \emph{assumption on intermediate fields}.

Theorem \ref{cdt} will follow from Theorem \ref{gencdt} by choosing $G=C_n$ and $\cF$ comprised of all generators of $G$. In the context of Theorem \ref{cdt}, we can verify all hypotheses of Theorem \ref{gencdt} unconditionally with $\tau=0$ and $\beta=\frac{1}{n-1}$; we will prove this in Section \ref{ublb}. Note that there is no ramification restriction when $n$ is a prime.

The proof of Theorem \ref{gencdt} is shown in Section \ref{pfgencdt}. It is based on the proof of \cite[Theorem 1.1]{PTW17} and \cite[Theorem 2.4]{TZ18} but introduces new ideas in Section \ref{genub} to overcome a difficulty when working over $k$ rather than $\Q$.

This ``meta theorem'' has strong hypotheses, some of which can be removed by adapting recent work of other authors. For example, Thorner and Zaman's innovative approach \cite{TZ19b} does not assume the strong Artin conjecture when working on certain families of extensions over $\Q$, but still requires an assumption on intermediate fields. Moreover, in a very recent preprint, Lemke Oliver, Thorner, and Zaman \cite{LOTZ} replace the ``subfield problem'' implicit as a hypothesis in our ``meta theorem'' by a different number field counting problem, and this has advantages in their work, when they deal with prime degree extensions and degree $n$ $S_n$-extensions of a fixed number field $k$.
But we still see a benefit in proving this ``meta theorem'' in some generality, because our approach introduces new ideas in Section \ref{genub} to overcome a difficulty that was inherent in \cite[Remark 6.5]{PTW17} when trying to adapt it to work over $k$ rather than $\Q$. Our new ideas relate to a clever observation of Kl\"{u}ners and Nicolae \cite{KN16} about whether a field is determined by its Artin $L$-functions, when working over a field other than $\Q$. After proving the ``meta theorem'', we specialize to a setting that is not treated in the other works mentioned above, and for which we can verify all of the hypotheses of our meta theorem unconditionally over any number field $k$.

\begin{subsection}{``Almost all'' fields in the family satisfy a zero-free region}

We begin the proof of the very general ``meta theorem'' (Theorem \ref{gencdt}).

We need the strong Artin conjecture for the effective Chebotarev density result stated in Theorem \ref{gencdt}. We note that Thorner and Zaman's innovative approach \cite{TZ19b} for families of fields over $\Q$ avoids this hypothesis by working directly with Dedekind zeta functions. In our approach, we require it only for the general discussion, and it is known in the case of Theorem \ref{cdt}.
 Our reference here is \cite{Mar03}.
\begin{conj}[strong Artin conjecture]\label{saconj}
Let $L/k$ be a Galois extension of number fields with $\mathrm{Gal}(L/k) \cong G$. Let $\rho$ be a complex representation of $G$ of dimension $m$. Then there exists an automorphic representation $\pi(\rho)$ of $\GL_m(\mathbb{A}_k)$ such that the $L$-function $L(s,\rho)$ and $L(s,\pi)$ agree at all but finitely many places. Moreover, if $\rho$ is irreducible, then $\pi$ is cuspidal.
\end{conj}
By \cite[Proposition 2.1]{Mar03}, if $\pi$ is cuspidal and $L(s,\rho)$ and $L(s,\pi)$ agree at all but finitely many places, then $L(s,\pi)=L(s,\rho)$.

The strong Artin conjecture is true for 1-dimensional representations by \cite{Art31}; nilpotent Galois extensions $L/k$ by \cite{AC89}, $A_4$ and $S_4$ by \cite{Lan80} and \cite{Tun81} respectively, dihedral groups by \cite{Lan80}. In particular, it is known for $G=C_n$, the cases considered in Theorem \ref{cdt}.

Fix a transitive subgroup $G \le S_n$. Let $\rho_1,\dots,\rho_s$ denote once and for all the nontrivial irreducible representations of $G$. For a Galois extension of number fields $L/k$ with $\mathrm{Gal}(L/k) \simeq G$, we have the decomposition of the Dedekind zeta function
\begin{equation}\label{zetarel}
\zeta_L(s)=\zeta_k(s)\prod_{j=1}^s L(s,\rho_j,L/k)^{m_j}
\end{equation}
where $m_j=\dim \rho_j$.

Fix $X \ge 1$. Define $\widetilde{Z}_n^\cF(k,G;X)$ to be the set of the Galois closures of fields $K$ over $k$ as $K$ varies in $Z_n^\cF(k,G;X)$. 
For $\rho_j$ a fixed nontrivial irreducible representation of $G$ as above, and a field $L \in \widetilde{Z}_n^\cF(k,G;X)$, we associate a cuspidal automorphic representation $\pi_{L,j}$ of $\GL(m_j)$ over $k$ via the strong Artin conjecture such that $L(s,\pi_{L,j})=L(s,\rho_j,L/k)$. We define $\cL_{n,j}(k,G;X)$ to be the set of cuspidal automorphic representations $\pi_{L,j}$ associated to $L$ and $\rho_j$, as $L$ varies over $\widetilde{Z}_n^\cF(k,G;X)$. Thus, for a family $\widetilde{Z}_n^\cF(k,G;X)$ and for each nontrivial irreducible representation $\rho_j$ of $G$ ($1 \le j \le s$), we have a set $\cL_{n,j}(k,G;X)$ of cuspidal automorphic representations.

For a family $Z_n^\cF(k,G)$ and any $X \ge 1$, we thus have the family $\widetilde{Z}_n^\cF(k,G;X)$; we will apply our zero density estimate of Theorem \ref{main} to the corresponding families $\cL_{n,j}(k,G;X)$ for each $j$.

We use the abbreviation $\cD(K/k)$ for $\mathrm{Nm}_{k/\Q}(\mathrm{Disc}(K/k))$. We recall the following lemmas.
\begin{lem}\label{69}
Let $K/k$ be a number field extension with $\mathrm{Gal}(\widetilde{K}/k) \cong G$ and let $H:=\mathrm{Gal}(\widetilde{K}/K)$. Let $\cP$ be a prime ideal in $k$ that is tamely ramified in $K$ and $\widetilde{K}$, and has an inertia group generated by $g \in G$. Then the power $\alpha$ such that $\cP^\alpha \| \mathrm{Disc}(K/k)$ is 
\begin{equation*}
[G:H]-\text{the number of orbits of $g$ acting on the cosets $G/H$}.
\end{equation*}
\end{lem}

Lemma \ref{69} is analogous to Lemma 6.9 in \cite{PTW17} and we refer the proof there. In \emph{loc. cit.} the proof is given for base field $\Q$ but it works also for a general base field $k$.

\begin{lem}\label{disccomp}
For each $K \in Z_n(k,G)$, we have
$D_{\widetilde{K}} \ll D_K^{|G|/2}.$
\end{lem}

\begin{proof}
Recall the formula for the relative discriminant
\begin{equation}\label{reldisc}
D_K=\cD(K/k)D_k^n.
\end{equation}
It suffices to prove
\begin{equation}\label{afterrd}
\cD(\widetilde{K}/k) \ll \cD(K/k)^{|G|/2}.
\end{equation}
For an extension $K/k$ with $\mathrm{Gal}(\widetilde{K}/k) \cong G$, all wildly ramified prime ideals of $k$ divide $|G|$. Thus, the total contributions to $\cD(K/k)$ and $\cD(\widetilde{K}/k)$ from wildly ramified prime ideals are at most a certain finite constant $C_G$ depending only on $G$. For tamely ramified prime ideals, we can bound their contributions to $\mathrm{Disc}(K/k)$ and $\mathrm{Disc}(\widetilde{K}/k)$ (and thus $\cD(K/k)$ and $\cD(\widetilde{K}/k)$) using Lemma \ref{69}. The formula (\ref{afterrd}) (hence Lemma \ref{disccomp}) follows.
\end{proof}

\begin{prop}\label{prop123}
Fix a family $Z_n^\cF(k,G)$ and the corresponding families of cuspidal automorphic representations $(\cL_{n,j}(k,G;X))_{X \ge 1}$ for $1 \le j \le s$. For each $X \ge 1$, for each $1 \le j \le s$, the family $\cL_{n,j}(k,G;X)$ satisfies

(1) for any $f \in \cL_{n,j}(k,G;X)$, $f$ satisfies the Ramanujan-Petersson conjecture;

(2) for any $f \in \cL_{n,j}(k,G;X)$, $|\mathrm{Cond}(f)| \le X^{|G|/2}$;

(3) for any $\ep>0$, $|\cL_{n,j}(k,G;X)| \ll X^{n|G|+\ep}$.
\end{prop}

\begin{proof}
For (1), the Ramanujan-Petersson conjecture is true for automorphic $L$-functions associated to Artin $L$-functions, once they are known to exist; see the comment below \cite[Theorem 5]{KM02}.

For (2), we use the argument below \cite[Lemma 6.1]{PTW17}. For an Artin $L$-function $L(s,\rho,L/k)$, if $F(\chi)$ denotes the Artin conductor of $\chi=\mathrm{Tr}(\rho)$, then the conductor of $L(s,\rho,L/k)$ is given by $A(\chi)=D_k^{\chi(1)}\mathrm{Nm}_{k/\Q}F(\chi)$. According to the multiplicativity relation $D_L=D_k\prod_{\chi_j}A(\chi_j)^{\chi_j(1)}$ for the conductors in the identity (\ref{zetarel}) and Lemma \ref{disccomp}, we see that for each $1 \le j \le s$, the conductors of $L(s,\rho_j,L/k)$ are bounded by $X^{|G|/2}$.

For (3), the bound follows from (2) and Remark \ref{autrepbd}.
\end{proof}

The next proposition is the key bridge between the assumption of Theorem \ref{gencdt} on intermediate fields, and the application of Theorem \ref{cdtzf}.

\begin{prop}\label{ch4}
Fix $G$ and the nontrivial irreducible representations $\rho_1,\dots,\rho_s$ of $G$. Let $Z_n^\cF(k,G)$ be a family of fields. Assume that there exists $0 \le \tau <\beta$ such that for all $X \ge 1$, for any $\ep_1>0$, for each $1 \le j \le s$ and any fixed $\pi \in \cL_{n,j}(k,G;X)$,
\begin{equation}\label{piljpi}
|\{ L \in \widetilde{Z}_n^\cF(k,G;X): \pi_{L,j}=\pi \}| \ll X^{\tau+\ep_1}. 
\end{equation}
For each $0<\ep<1$, there exists $\delta$ (chosen in (\ref{smalldelta})) such that for every $X \ge 1$, in the set $Z_n^\cF(k,G;X)$,
there are $\ll_{\ep} X^{\tau+\ep}$ fields $K$ such that $\zeta_{\widetilde{K}}/\zeta_k$ could have a zero in the region 
\begin{equation}\label{lkzf2}
[1-\delta,1] \times [-Q(\log Q)^{24},Q(\log Q)^{24}]
\end{equation}
where $Q=X^{|G|/2}$.
\end{prop}
Proposition \ref{ch4} shows that for any $X \ge 1$ and $0<\ep<1$, there are at most $\ll X^{\tau+\ep}$ fields $K$ such that $L(s,\rho_j,\widetilde{K}/k)$ may have a zero in the region (\ref{lkzf2}). Since there are by assumption $\gg X^\beta$ fields in the family $Z_n^\cF(k,G;X)$ and $\beta>\tau$, we know that assuming the corresponding conditions for the family $Z_n^\cF(k,G;X)$, ``almost all'' fields in the family satisfy the effective Chebotarev density theorem, so that we will be able to deduce Theorem \ref{gencdt}.

The hypothesis (\ref{piljpi}) is very strong; we will show in the next section how it can follow from the assumption in Theorem \ref{gencdt} on intermediate fields, where we will face new difficulties since we work over $k$ rather than $\Q$.

\begin{proof}[Proof of Proposition \ref{ch4}]
We will prove this via an application of the zero density estimate in Theorem \ref{main}. 

For each fixed $1 \le j \le s$, we have Proposition \ref{prop123} for each family in $(\cL_{n,j}(k,G;X))_{X \ge 1}$. Taking parameters $A_j=A=|G|/2, d_j=d=n|G|+1, q=X$, we are able to apply Theorem \ref{main} to the family $\cL_{n,j}(k,G;X)$ for all $X \ge 1$ (as long as the strong Artin conjecture is known or assumed). Then for any $\frac34 \le \alpha_j \le 1$ and $T_j \ge 2$, Theorem \ref{main} shows
\begin{equation*}
\sum_{\pi \in \cL_{n,j}(k,G;X)} N(\pi; \alpha_j,T_j) \ll (X^{c_{j,1}}T_j^{c_{j,2}})^{1-\alpha_j}
\end{equation*} 
where 
\begin{equation}\label{cj1cj2}
 c_{j,1}=c_1=2d+4nA+\frac{A}{2}+1+\ep, \ \ c_{j,2}=c_2=\frac{nn_k}{2}+3+\ep.
 \end{equation}
Let $T_j=T=Q(\log Q)^{24}$ where $Q=X^{|G|/2}$.
 We fix $\alpha_j$ such that
\begin{equation}\label{alphaj}
(c_{j,1}+\frac{|G|}{2}c_{j,2})(1-\alpha_j)=\frac{\ep}{2}.
\end{equation}
Since $\ep<1$, we have $\alpha_j \ge \frac34$.
Thus,
\begin{equation*}
\sum_{\pi \in \cL_{n,j}(k,G;X)} N(\pi;\alpha_j,T) \ll X^{(c_{j,1}+\frac{|G|}{2}c_{j,2})(1-\alpha_j)}(\log X)^{24c_{j,2}(1-\alpha_j)} \ll X^{\frac34 \ep}.
\end{equation*}
Now we combine the zero density estimate for all $1 \le j \le s$. Since $c_{j,1}$ and $c_{j,2}$ do not depend on $j$, $\alpha_j$ does not either and we set $\alpha=\alpha_j$. For an $L$-function $\cL$, we define $N(\cL;\alpha,T)$ to be the number of zeros of $\cL$ in the region $\{\sigma+it: \alpha \le \sigma \le 1, \ |t| \le T \}$. Then for each $X \ge 1$, using (\ref{zetarel}) and the strong Artin conjecture, followed by the hypothesis (\ref{piljpi}), we see
\begin{eqnarray*}
\nonumber \sum_{L \in \widetilde{Z}_n^\cF(k,G;X)} N(\zeta_L/\zeta_k;\alpha,T) &=& \sum_{L \in  \widetilde{Z}_n^\cF(k,G;X)} \sum_{j=1}^s m_jN(L(s,\rho_j,L/k);\alpha,T) \\
\nonumber &=& \sum_{L \in  \widetilde{Z}_n^\cF(k,G;X)} \sum_{j=1}^s m_jN(L(s,\pi_{L,j});\alpha,T) \\
\nonumber &=& \sum_{j=1}^s m_j \sum_{\pi \in \cL_{n,j}(k,G;X)} N(\pi;\alpha,T)\sum_{\substack{L \in  \widetilde{Z}_n^\cF(k,G;X) \\ \pi_{L,j}=\pi}} 1 \\
& \ll & \sum_{j=1}^s m_jX^{\tau+\frac{\ep}{4}} \sum_{\pi \in \cL_{n,j}(k,G;X)} N(\pi;\alpha,T) \ll X^{\tau+\ep}.
\end{eqnarray*}
We set 
\begin{equation}\label{smalldelta}
\delta=1-\alpha=1-\alpha_j=\frac{\ep}{2(c_{j,1}+\frac{|G|}{2}c_{j,2})}.
\end{equation}
Then Proposition \ref{ch4} follows, with this choice of $\delta$.
\end{proof}
\end{subsection}

\begin{subsection}{Translation of (\ref{piljpi}) to the assumption on intermediate fields}\label{genub}

In this section, we deduce the condition (\ref{piljpi}) from a condition of counting intermediate fields, assumed as a hypothesis in Theorem \ref{gencdt}. This adapts ideas of \cite{PTW17} over $\Q$ to our new setting over a field $k$.

We first recall earlier work over $\Q$. In \cite{PTW17}, it is shown that (\ref{piljpi}) can be controlled by counting the number of fields $K$ in $Z_n^{\cF}(k,G;X)$ such that $\widetilde{K}^{\ker(\rho_j)}=F$, for certain fields $F$. Their argument uses Proposition 6.3 of \cite{PTW17}, which shows that for a fixed representation $\rho$ of $G \le S_n$ and two fields $L_1,L_2$ with $\mathrm{Gal}(L_1/\Q) \cong \mathrm{Gal}(L_2/\Q) \cong G$, if $L(s,\rho,L_1/\Q)=L(s,\rho,L_2/\Q)$, then $L_1^{\ker(\rho)}=L_2^{\ker(\rho)}$. Simply replacing $\Q$ by an arbitrary fixed number field $k$ in this statement can be false (see the remark after \cite[Proposition 3]{KN16}). Here, we instead use a novel argument that addresses subtle new issues that did not arise when $k=\Q$ in \cite{PTW17}. The idea is motivated by Proposition 4 and Theorem 6 in \cite{KN16}. 

Fix $k/\Q$ of degree $n_k$. For a number field extension $K/k$ of degree $n$, we denote by $\widetilde{K}$ the Galois closure of $K$ over $k$ and denote by $\widetilde{\widetilde{K}}$ the Galois closure of $\widetilde{K}$ over $\Q$, with $\mathrm{Gal}(\widetilde{\widetilde{K}}/\Q) \cong :G'$.

We consider the Galois closure of $K$ over $\Q$. It contains $\widetilde{K}$ hence contains $\widetilde{\widetilde{K}}$. Therefore, the Galois closure of $K$ and $\widetilde{K}$ over $\Q$ coincide, and $G'$ is a transitive subgroup of $S_{nn_k}$. 

We fix a nontrivial irreducible representation $\rho_j$ of $G$.
Then $\ker(\rho_j)$ is a proper subgroup of $G$. Let $F:=\widetilde{K}^{\ker(\rho_j)}$, $U:=\mathrm{Gal}(\widetilde{\widetilde{K}}/k)$, $V:=\mathrm{Gal}(\widetilde{\widetilde{K}}/{\widetilde{K}})$. Since $\rho_j$ is nontrivial, $F$ is a nontrivial extension of $k$. We identify the groups $G$ and $U/V$. Then we have the following lattice of fields.
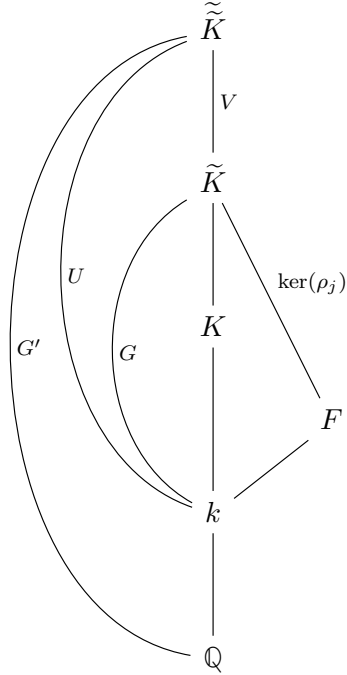
\begin{equation*}
\begin{tikzcd}
 \widetilde{\widetilde{K}} \ar[dd, dash, "V"] \ar[dddddd, dash, "U", bend right=70] \ar[dddddddd, dash, "G'", bend right=80]  \\
 & \\
 \widetilde{K} \ar[dd, dash] \ar[dddd, dash, "G", bend right=60] \ar[dddr, dash, "\ker(\rho_j)"] \\
 & \\
 K \ar[dd, dash] \\
 & F \ar[dl, dash] \\
 k \ar[dd, dash] \\
 & \\
 \Q
\end{tikzcd}
\end{equation*}

Observe that
\begin{equation*}
\widetilde{\rho_j}(\sigma):=\rho_j(\sigma V)
\end{equation*}
is a representation of $U$. Let $\psi_j=\mathrm{Ind}_{U}^{G'} \widetilde{\rho_j}$ be the induced representation. Then we have the following result.
\begin{lem}\label{kpsif}
The field $\widetilde{\widetilde{K}}^{\ker(\psi_j)}$ contains $F$.
\end{lem}
\begin{proof}[Proof of Lemma \ref{kpsif}]
We have
\begin{equation*}
\ker(\psi_j)=\bigcap_{\sigma \in G'} \sigma \ker(\widetilde{\rho_j}) \sigma^{-1}.
\end{equation*}
Taking $\sigma$ to be the identity, we obtain that $\ker(\psi_j) \le \ker(\widetilde{\rho_j})$. Then we have
\begin{equation*}
F=\widetilde{\widetilde{K}}^{\ker(\widetilde{\rho_j})} \subseteq \widetilde{\widetilde{K}}^{\ker(\psi_j)}.
\end{equation*}
\end{proof}

Let $\chi_j=\mathrm{tr}(\rho_j)$ and $\widetilde{\chi_j}=\mathrm{tr}(\widetilde{\rho_j})$. Then by convention $L(s,\rho_j,\widetilde{K}/k)=L(s,\chi_j,\widetilde{K}/k)$. From \cite[p. 297]{Art31},
\begin{equation}\label{art31good}
L(s,\rho_j,\widetilde{K}/k)=L(s,\chi_j,\widetilde{K}/k)=L(s,\widetilde{\chi_j},\widetilde{\widetilde{K}}/k)=L(s,\psi_j,\widetilde{\widetilde{K}}/\Q).
\end{equation}

We are ready to show how to deduce condition (\ref{piljpi}) from the assumption in Theorem \ref{gencdt} on intermediate fields, which we can now formulate precisely.
Let $Z_n^\cF(k,G)$ be a family of fields as defined in (\ref{fml2}).
Fix $X \ge 1$.  
For each $1 \le j \le s$, define a set of fields
$\mathfrak{F}_j$ as
\begin{equation*}
\mathfrak{F}_{j}:=\{ F/k: F=\widetilde{K}^{\ker(\rho_j)} \text{ for some field } K \in Z_n^\cF(k,G;X) \}
\end{equation*}
 and define
\begin{eqnarray}\label{mfrakf}
\nonumber \mathfrak{F}:=\bigcup_{1 \le j \le s} \mathfrak{F}_j=\{ F/k &:& F=\widetilde{K}^{\ker(\rho_j)} \text{ for some field } K \in Z_n^\cF(k,G;X) \\
& &\text{ and some nontrivial irreducible representation } \rho_j \text{ of } G \}.
\end{eqnarray}
This definition completes the description of the assumption in Theorem \ref{gencdt} on intermediate fields. Our goal is to derive (\ref{piljpi}) under the assumption that, for each $1 \le j \le s$, there exists $0 \le \tau$ such that for any field $F \in \mathfrak{F}_j$, and any $\ep_1>0$, there are $\ll X^{\tau+\ep_1}$ fields $K$ in $Z_n^\cF(k,G;X)$ such that $\widetilde{K}$ is an extension of $F$.

Fix $1 \le j \le s$ and a cuspidal automorphic representation $\pi_j \in \cL_{n,j}(k,G;X)$. In the context of (\ref{piljpi}), one field $L \in \widetilde{Z}_n^\cF(k,G;X)$ corresponds to up to $C(G)$ fields $K \in Z_n^\cF(k,G;X)$ by Galois theory, where $C(G)$ is a positive integer depending only on $G$. Thus, it suffices to show that for any $\ep_1>0$,
\begin{equation}\label{kpi}
|\{ K \in Z_n^\cF(k,G;X): L(s,\rho_j,\widetilde{K}/k)=L(s,\pi_j)\}| \ll X^{\tau+\ep_1}.
\end{equation}
If $K_1,K_2 \in Z_n^\cF(k,G;X)$ are two fields such that $L(s, \rho_j, \widetilde{K_1}/k)=L(s, \rho_j, \widetilde{K_2}/k)=L(s,\pi_j)$, then $L(s,\psi_j,\widetilde{\widetilde{K_1}}/\Q)=L(s,\psi_j,\widetilde{\widetilde{K_2}}/\Q)=L(s,\pi_j)$ by (\ref{art31good}). By Proposition 6.3 of \cite{PTW17}, we have $\widetilde{\widetilde{K_1}}^{\ker{\psi_j}}=\widetilde{\widetilde{K_2}}^{\ker{\psi_j}}$; this now applies since we have defined the problem over $\Q$. Thus, it follows that
$M_{\pi_j}:=\widetilde{\widetilde{K_1}}^{\ker{\psi_j}}=\widetilde{\widetilde{K_2}}^{\ker{\psi_j}}$ is a field extension of $k$ depending only on $\pi_j$. Denote $F_i={\widetilde{K_i}}^{\ker{\rho_j}}$ for $i=1,2$. Then $F_1$ and $F_2$ are extensions of $k$, and both are contained in $M_{\pi_j}$ by Lemma \ref{kpsif}. 
We have just proved that for any field $K \in Z_n^\cF(k,G;X)$ such that $L(s,\rho_j,\widetilde{K}/k)=L(s,\pi_j)$, we have $k \subset {\widetilde{K}}^{\ker{\rho_j}} \subseteq M_{\pi_j}$. 
Let $\cM_{\pi_j}$ denote the set of intermediate fields between $M_{\pi_j}$ and $k$. Define $\mathfrak{F}_{\pi_j} \subset \mathfrak{F}_j$ as
\begin{equation*}
\mathfrak{F}_{\pi_j}:=\{ F/k: F=\widetilde{K}^{\ker(\rho_j)} \text{ for some field } K \in Z_n^\cF(k,G;X) \text{ satisfying } L(s,\rho_j,\widetilde{K}/k)=L(s,\pi_j) \}.
\end{equation*}
 Based on the definition of $\mathfrak{F}_{\pi_j}$, the inequality $|\mathfrak{F}_{\pi_j}| \le |\cM_{\pi_j}|$ holds. Moreover, we show that $|\cM_{\pi_j}| \ll 1$ as follows. For any field $K \in Z_n^\cF(k,G)$, $[\widetilde{K}:\Q] \le (n!)n_k$. Thus, $|\cM_{\pi_j}|$ can be trivially bounded (via Galois theory) by the number of subgroups of $S_{(n!)n_k}$, which is $\ll 1$. Therefore, $|\mathfrak{F}_{\pi_j}| \ll 1$.

To bound the left-hand side of (\ref{kpi}), we distinguish the fields $K$ by the fields $\widetilde{K}^{\ker{\rho_j}}$. Fix $1 \le j \le s$ and $\pi_j \in \cL_{n,j}(k,G;X)$ as above. Direct computation shows that, for any $\ep_1>0$, assuming the hypothesis in Theorem \ref{gencdt} on intermediate fields,
\begin{eqnarray*}
& & |\{ K \in Z_n^\cF(k,G;X): L(s,\rho_j,\widetilde{K}/k)=L(s,\pi_j)\}| \\
&=& \sum_{F \in \mathfrak{F}_j} |\{ K \in Z_n^\cF(k,G;X): \widetilde{K}^{\ker{\rho_j}}=F, \ L(s,\rho_j,\widetilde{K}/k)=L(s,\pi_j)\}| \\
&=& \sum_{F \in \mathfrak{F}_{\pi_j}} |\{ K \in Z_n^\cF(k,G;X): \widetilde{K}^{\ker{\rho_j}}=F, \ L(s,\rho_j,\widetilde{K}/k)=L(s,\pi_j)\}| \\
& \le & \sum_{F \in \mathfrak{F}_{\pi_j}} |\{ K \in Z_n^\cF(k,G;X): \widetilde{K} \text{ is an extension of } F \}| \ll \sum_{F \in \mathfrak{F}_{\pi_j}} X^{\tau+\ep_1} = |\mathfrak{F}_{\pi_j}|X^{\tau+\ep_1} \ll  X^{\tau+\ep_1}.
\end{eqnarray*}
This finishes the proof of (\ref{kpi}) under the assumption in Theorem \ref{gencdt}. Hence, (\ref{piljpi}) holds under the assumption in Theorem \ref{gencdt} on intermediate fields.

\begin{rem}\label{specialub}
 If $\rho_j$ is faithful and $L(s, \rho_j, \widetilde{K_1}/k)=L(s, \rho_j, \widetilde{K_2}/k)$, our conclusion would be $\widetilde{\widetilde{K_1}}=\widetilde{\widetilde{K_2}}$ by \cite[Theorem 6]{KN16}, leading to finitely many choices of the extension $K/k$ that share the same $L$-function. The upper bound would then be $\ll 1$ ($\tau=0$). This is the case when $G=C_n$, $n$ prime. For $G=C_n$, $n$ composite, we show $\tau=0$ in Section \ref{ublb} by more delicate means.
\end{rem}

\end{subsection}

\begin{subsection}{Proof of Theorem \ref{gencdt}}\label{pfgencdt}
For any $\ep>0$, we let $\kappa=\delta/8=\frac{\ep}{2(c_{j,1}+\frac{|G|}{2}c_{j,2})}$, where $\delta$ is defined in (\ref{smalldelta}) and $c_{j,1},c_{j,2}$ defined in (\ref{cj1cj2}). Let $K \in Z_n^\cF(k,G;X)$ be a field such that $\zeta_{\widetilde{K}}/\zeta_k$ has no zeros in the region (\ref{lkzf2}). The hypothesis of Theorem \ref{gencdt} and Section \ref{genub} show that (\ref{piljpi}) holds, so by Proposition \ref{ch4}, all but $\ll X^{\tau+\ep}$ fields in the family $Z_n^\cF(k,G;X)$ satisfy this property, for any $0<\ep<1$. For such a $K$, define $E(x):=x^{-\kappa}+T^{-\frac{1}{24}}e^{-\frac{1}{24}\sqrt{c_4(\log x)/n_{\widetilde{K}}}}+T^{-\frac{1}{24}}e^{-\frac{1}{24}\frac{c_4 \log x}{\log D_{\widetilde{K}}}}$ to be the term appearing in Theorem \ref{cdtzf}, where $T=Q(\log Q)^{24}$, $Q=X^{|G|/2}$. Since $T>Q=X^{|G|/2} \gg D_{\widetilde{K}}$ by Lemma \ref{disccomp}, 
\begin{equation*}
E(x) \ll E'(x):=x^{-\kappa}+D_{\widetilde{K}}^{-\frac{1}{24}}e^{-\frac{1}{24}\sqrt{c_4(\log x)/n_{\widetilde{K}}}}+D_{\widetilde{K}}^{-\frac{1}{24}}e^{-\frac{1}{24}\frac{c_4 \log x}{\log D_{\widetilde{K}}}}.
\end{equation*}
It suffices to prove that for another absolute constant $c_3>0$, and for $K$ as above,
\begin{equation*}
E'(x) \ll
\begin{cases}
x^{-\kappa} & \text{if } (\log D_{\widetilde{K}})^{2/\kappa} \le x<D_{\widetilde{K}}^{1/(24\kappa)}, \\
 \frac{1}{\exp(c_3(\log x)^{1/2}|G|^{-1/2}n_k^{-1/2})} & \text{if } x \ge D_{\widetilde{K}}^{1/(24\kappa)}.
\end{cases}
\end{equation*}
For $(\log D_{\widetilde{K}})^{2/\kappa} \le x<D_{\widetilde{K}}^{1/(24\kappa)}$, direct computation shows $E'(x) \ll x^{-\kappa}$. For $\frac{1}{24\kappa}\log D_{\widetilde{K}} \le \log x \le c_4^{-1}|G|n_k(\log D_{\widetilde{K}})^2$, computation shows $E'(x) \ll D_{\widetilde{K}}^{-\frac{1}{24}} \ll e^{-\frac{1}{24}\sqrt{c_4(\log x)|G|^{-1}n_k^{-1}}}$. For $\log x \ge c_4^{-1}|G|n_k(\log D_{\widetilde{K}})^2$, we have $E'(x) \ll e^{-\frac{1}{24}\sqrt{c_4(\log x)|G|^{-1}n_k^{-1}}}+e^{-\frac{c_4\log x}{24\log D_{\widetilde{K}}}} \ll e^{-\frac{1}{24}\sqrt{c_4(\log x)|G|^{-1}n_k^{-1}}}$. Theorem \ref{gencdt} then follows with $c_3=\frac{1}{24}\sqrt{c_4}$.
\end{subsection}

\begin{subsection}{Proof of Theorem \ref{cdt}}\label{ublb}

To deduce Theorem \ref{cdt} from the very general ``meta theorem'' of Theorem \ref{gencdt}, we choose $G=C_n$ and $\cF$ comprising of all generators of $G$. The strong Artin conjecture is true for $G$ by \cite{Art31}. Theorem \ref{gencdt} directly gives the result of Theorem \ref{cdt} once we prove that $\tau=0$ and $\beta=\frac{1}{n-1}$ in this specific case. To do so, we follow \cite{Wr89} and \cite{PTW17}, and insert our ramification restriction in the argument.

By class field theory, $G$-extensions $K$ of $k$ correspond to open subgroups $U$ of the id\`{e}le class group $C_k=I_k/k^\times$ such that $\mathrm{Gal}(K/k) \cong C_k/U$, where $I_k$ is the group of id\`{e}les of $k$. Moreover, the relative discriminant of $K/k$ may be expressed in terms of conductors of characters of $C_k/U$.

Define the Dirichlet series
\begin{equation*}
A(s)=\sum_{m \ge 1} \frac{a_m}{m^s}=\sum_{K \in Z_n^\cF(k,G)} \frac{1}{\cD(K/k)^s}
\end{equation*}
where $\cD(K/k)=\mathrm{Nm}_{k/\Q}(\mathrm{Disc}(K/k))$ and $a_m$ is the number of fields $K \in Z_n^\cF(k,G)$ up to isomorphism such that $\cD(K/k)=a_m$,
and define the generating series of conductors by
\begin{equation*}
F(s)=\sum_{\chi \in C_n(C_k)} \Phi_G(\chi,s), \text{ where } \Phi_G(\chi,s)=\prod_{0 \le a <n} \frac{1}{\mathrm{Cond}(\chi^a)^s}
\end{equation*}
where $C_n(C_k)$ is the group of all continuous characters $\chi$ of $C_k$ (regarded as characters of $I_k$ that are trivial on $k^\times$) such that $\chi_v$ is trivial for all but one place $v_0$ and $\chi_{v_0}^n=1$, $\chi_{v_0}^m \neq 1$ for any $1 \le m<n$. 
Then
\begin{equation*}
F(s)=\Phi(n)A(s)
\end{equation*}
since one group $U$ corresponds to $\Phi(n)$ characters $\chi \in C_n(C_k)$ such that $\ker(\chi)=U$. By transferring to the context of conductors, we will have Euler products. We let $S$ be a finite set of places of $k$ that includes all infinite places, finite places dividing $n$, and finite places of $S$ that generate the class group of $k$. Let $C_{k,S}$ be the group of id\`{e}les which have components in $\cO_v^\times$ for all places $v \notin S$. Let $\cO_S=C_{k,S} \cap k^\times$ be the ring of $S$-integers of $k$ (elements of $k$ with non-negative valuation at all places not in $S$), then $\cO_S$ has class number 1. By \cite[Lemma 2.8]{Woo10}, we have an isomorphism
\begin{equation*}
C_{k,S}/\cO_S^\times \cong C_k.
\end{equation*}
We let $\cA_n(S)$ be a fixed finite set of representatives of $\cO_S/\cO_S^n$ and denote $a_n(S)=|\cA_n(S)|$. To sieve out those characters that vanish on $k^\times$, we use the delta function
\begin{equation*}
\delta_n(\chi)=\frac{1}{a_n(S)} \sum_{\ep \in \cA_n(S)} \chi(\ep),
\end{equation*}
which is 1 for $\chi(\cA_n(S))=1$ and is 0 otherwise. Then
\begin{equation}\label{newdelta}
F(s)=\sum_{\chi \in C_n(C_{k,S})} \delta_n(\chi) \Phi_G(\chi,s).
\end{equation}
We interchange the sum for $\chi$ and for $\ep$ in the sum (\ref{newdelta}) and obtain
\begin{equation*}
F(s)=\sum_{\ep \in \cA_n(S)} F(s,\ep)
\end{equation*}
where
\begin{equation*}
F(s,\ep)=\sum_{\chi \in C_n(C_{k,S})} \chi(\ep) \Phi_G(\chi,s).
\end{equation*}
The Euler factorization of $F(s,\ep)$ is
\begin{equation*}
F(s,\ep)=\prod_{v \in S} \sum_{\chi_v \in C_n(k_v^\times)} \chi_v(\ep_v) \Phi_G(\chi_v,s) \times \prod_{v \notin S} \sum_{\chi_v \in C_n(\cO_v^\times)} \chi_v(\ep_v) \Phi_G(\chi_v,s) =: \prod_v F_v(s,\ep).
\end{equation*}
There are finitely many places in $S$ and for each $v \in S$, there are finitely many characters of $k_v^\times$ such that $\chi_v^n=1$. As a consequence, $\displaystyle{\prod_{v \in S} F_v(s,\ep)}$ is a polynomial of $s$. We denote the polynomial as $P(s,\ep)$.
Now it suffices to consider $F_v(s,\ep)$ for $v \notin S$.

Our ramification restriction enforces that for $\chi_v \in C_n(\cO_v^\times)$, $\chi_v(\cO_v^\times)$ is trivial or generates $G=C_n$. By Proposition 4.3 of \cite{Wr89}, for $v \notin S$ such that $\mathrm{Nm}(v) \equiv 1 (\mathrm{mod} \ n)$, we have
\begin{equation*}
F_v(s,\ep)=1+\Phi(n)\mathrm{Nm}(v)^{-(n-1)s}+O(\mathrm{Nm}(v)^{-2(n-1)s}).
\end{equation*}
For $v \notin S$ such that $\mathrm{Nm}(v) \not\equiv 1 (\mathrm{mod} \ n)$, we have $F_v(s,\ep)=1$.
Thus,
\begin{equation*}
F(s)=\sum_{\ep \in \cA_n(S)} F(s,\ep)=\sum_{\ep \in \cA_n(S)} P(s,\ep) \prod_{\substack{v \notin S \\ \mathrm{Nm}(v) \equiv 1 (\mathrm{mod} \ n)}}(1+\Phi(n)\mathrm{Nm}(v)^{-(n-1)s}+O(\mathrm{Nm}(v)^{-2(n-1)s})).
\end{equation*}
Therefore, $a_m \ll_n \Phi(n)^{n\omega(m)} \ll_{n,\ep} m^\ep$ for any $\ep>0$. 

We are ready to verify the assumption in Theorem \ref{gencdt} on intermediate fields for our family $Z_n^\cF(k,C_n)$ with $\tau=0$.
For a fixed nontrivial extension $F$ over $k$, and for $K \in Z_n^\cF(k,C_n)$ such that $\widetilde{K}$ is an extension of $F$, we know by Lemma \ref{69} that every prime $\cP \subset \cO_k$ dividing $\mathrm{Disc}(K/k)$ must divide $\mathrm{Disc}(F/k)$ and the number $\alpha$ such that $\cP^\alpha \| \mathrm{Disc}(K/k)$ is uniquely determined. Thus, the contribution to $\cD(K/k)$ from tamely ramified primes not in $S$ is fixed. Since $a_m \ll_{n,\ep} m^\ep$, we learn that the assumption holds with $\tau=0$.

Now we consider the lower bound for the number of fields in our family. We write, for $\Re(s)>1$,
\begin{equation*}
B(s)=\prod_{\chi} \left( \sum_{\mathfrak{m}} \frac{\chi(\mathrm{Nm(\mathfrak{m})})}{(\mathrm{Nm(\mathfrak{m})})^s} \right)=\prod_{\chi} \left( \prod_v (1-\chi(\mathrm{Nm}(v))(\mathrm{Nm}(v))^{-s})^{-1} \right),
\end{equation*}
where $\chi$ runs over all Dirichlet characters modulo $n$, $\mathfrak{m}$ runs over all ideals of $\cO_k$ coprime to the prime ideals in $S$, and $v$ runs over all prime ideals of $\cO_k$ not in $S$. Then $B(s)$ has an analytic continuation to $\C$ with the only simple pole at $s=1$ (see e.g., \cite[p.129]{IK04}). Moreover,
\begin{equation*}
B(s)=\prod_v \mu_v(s)^{-1}
\end{equation*}
where
\begin{eqnarray*}
\mu_v(s)&=&1-\sum_\chi \chi(\mathrm{Nm}(v))\mathrm{Nm}(v)^{-s}+O(\mathrm{Nm}(v)^{-2s}) \\
&=& \begin{cases}
1-\Phi(n)\mathrm{Nm}(v)^{-s}+O(\mathrm{Nm}(v)^{-2s}) & \text{ if } \mathrm{Nm}(v) \equiv 1 (\mathrm{mod} \ n), \\
1+O(\mathrm{Nm}(v)^{-2s}) & \text{ if }\mathrm{Nm}(v) \not\equiv 1 (\mathrm{mod} \ n).
\end{cases}
\end{eqnarray*}
Then $\frac{F(s)}{B((n-1)s)}$ is holomorphic on $\Re(s)>\frac{1}{2(n-1)}$. Thus, $F(s)$ can be analytically continued to $\Re(s)>\frac{1}{2(n-1)}$ with a simple pole at $\frac{1}{n-1}$. Moreover, $F(s)$ inherits a standard convexity estimate from $B(s)$; see, e.g., \cite[Lemma 5.2, Theorem 5.23]{IK04}. So, by a standard Tauberian theorem (see, e.g., \cite[Section 6.4]{Nar00}), we have
\begin{equation*}
|Z_n^\cF(k,G;X)|=c_nX^{\frac{1}{n-1}}+o(X^{\frac{1}{n-1}}),
\end{equation*}
for some constant $c_n$. This proves $\beta=\frac{1}{n-1}$ in the notation of Theorem \ref{gencdt}.

\end{subsection}

\end{section}

\begin{section}{Application to bounding $\ell$-torsion in class groups}

Let $K$ be a number field. Recall the $\ell$-torsion subgroup
\begin{equation*}
\Cl_K[\ell]=\{ [\mathfrak{a} ] \in \Cl_K: [\mathfrak{a} ]^\ell=\Id \}.
\end{equation*}
A trivial bound for the $\ell$-torsion subgroup derives from any upper bound for the class group. In particular,
$|\Cl_K[\ell]| \le |\Cl_K| \ll_{d,\ep} D_K^{1/2+\ep}$,
where $d=[K:\Q]$, by the Minkowski bound.
On the other hand,  the $\ell$-torsion conjecture states that $|\Cl_K[\ell]| \ll_{d,\ell,\ep} D_K^{\ep}$, for any $\ep>0$. 

In \cite[Proposition 3.1]{EV07}, Ellenberg and Venkatesh show that under the Generalized Riemann Hypothesis, for any $\ep>0$, one obtains
\begin{equation}\label{grh}
|\Cl_K[\ell]| \ll_{d,\ell,\ep} D_K^{\frac12-\frac{1}{2\ell(d-1)}+\ep}.
\end{equation}
We use our effective Chebotarev density theorem to prove the bound (\ref{grh}) unconditionally for almost all fields in the family $Z_n^\ast(k;X)=Z_n^\cF(k,C_n;X)$ considered in Theorem \ref{ltorsion}. 
Our method is analogous to that of Pierce, Turnage-Butterbaugh, and Wood; see \cite{PTW17}. 

We first state a ``meta theorem'' based on the very general setting summarized in Theorem \ref{gencdt}.
\begin{thm}\label{genltor}
Let $Z_n^\cF(k,G), \tau, \beta$ satisfy the conditions assumed in Theorem \ref{gencdt}. Then for every $\ell \in \Z_{\ge 1}$ and every $X \ge 1$, $0<\ep<1$, aside from at most $\ll_{\cF,\ep} X^{\tau+\ep}$ possible exceptions, each field $K \in Z_n^\cF(k,G;X)$ has the property that
\begin{equation*}
|\Cl_K[\ell]| \ll_{n,n_k,D_k,\ell,|G|,\ep} D_K^{\frac12-\frac{1}{2\ell(n-1)}+\ep}.
\end{equation*}
\end{thm}

Theorem \ref{ltorsion} follows from Theorem \ref{genltor} by choosing $G=C_n$ and $\cF$ comprising of all generators of $G$, in which case $\tau=0,\beta=\frac{1}{n-1}$, and we know Theorem \ref{gencdt} holds unconditionally (Theorem \ref{cdt}).

Our approach to prove Theorem \ref{genltor} is similar to \cite[Theorem 7.2]{PTW17}.
We need the following lemma of Ellenberg and Venkatesh.

\begin{lemma}[Lemma 2.3 of \cite{EV07}]\label{ev07}
Suppose $K/k$ is an extension of number fields of degree $d$, let $\ell$ be a positive integer, and let $\delta<\frac{1}{2\ell(d-1)}$. Suppose that $\{p_1,\dots, p_M \}$ are prime ideals of $k$ of norm at most $\cD(K/k)^\delta=\mathrm{Nm}(\mathrm{Disc}(K/k))^\delta$ that are unramified and are not extensions of prime ideals from any proper subfield of $K$ containing $k$. Then
\begin{equation*}
|\mathrm{Cl}_K[\ell]| \ll_{[K:\Q],\ep,\ell} D_K^{\frac12+\ep}M^{-1}.
\end{equation*}
\end{lemma}
If we have $M$ prime ideals of $k$ of norm at most $\mathrm{Nm}(\mathrm{Disc}(K/k))^\delta$ that are unramified and split completely in $K$, then the condition in Lemma \ref{ev07} is satisfied and the trivial $\ell$-torsion bound for $K$ is improved by a factor $M^{-1}$. 

Theorem \ref{genltor} can be derived by combining our effective Chebotarev density Theorem \ref{gencdt} and Lemma \ref{ev07}. In particular, we prove the following proposition, analogous to \cite[Corollary 3.16]{PTW17}.

\begin{prop}\label{manyideals}
Let $Z_n^\cF(k,G), \tau, \beta$ satisfy the conditions assumed in Theorem \ref{gencdt}. Assume that there exists $\tau \ge 0$ such that for every $X \ge 1$, any $\ep_1>0$, and for certain nontrivial field extensions $F/k$ (described in Section \ref{genub}), there are $\ll X^{\tau+\ep_1}$ fields in $Z_n^\cF(k,G;X)$ that are extensions of $F$. 
Then for any $\sigma>0$ and any $0<\ep<1$, there exists a constant $D_0$ such that except for at most $\ll X^{\tau+\ep}$ fields, every field $K \in Z_n^\cF(k,G;X)$ with $D_K \ge D_0$ has the property that for any fixed conjugacy class $\cC$ of $G$,
\begin{equation*}
\pi_\cC(\cD(K/k)^\sigma,\widetilde{K}/k) \gg_{|G|,n,n_k,D_k,\sigma} \frac{D_K^\sigma}{\log D_K}.
\end{equation*}
\end{prop}

\begin{proof}
Let $\kappa$ be the parameter as in Theorem \ref{gencdt}.
For fixed $\sigma>0$ and any $\ep^{'}>0$, there is a threshold $D_0^{'}$ such that for $D_K \ge D_0^{'}$,
\begin{equation*}
\cD(K/k)^{\ep^{'}} \ge (\log D_{\widetilde{K}})^{2/\kappa}.
\end{equation*}
This is clear once we have the formula of the relative discriminant (\ref{reldisc})
and Lemma \ref{disccomp}. By Theorem \ref{gencdt}, for every $X \ge 1$ and any $0<\ep<1$, aside from $\ll X^{\tau+\ep}$ exceptions, every field $K \in Z_n^\cF(k,G;X)$ with $D_K \ge D_0^{'}$ satisfies 
\begin{equation*}
\left| \pi_\cC(\cD(K/k)^\sigma,\widetilde{K}/k)-\frac{|\cC|}{|G|}\mathrm{Li}(\cD(K/k)^\sigma) \right| \le \frac{|\cC|}{|G|} \frac{\cD(K/k)^\sigma}{\exp(c_3(\sigma \log \cD(K/k))^{1/2}|G|^{-1/2}n_k^{-1/2})}.
\end{equation*}
There exists a threshold $D_1^{'}$ such that for $D_K \ge D_1^{'}$, 
\begin{equation*}
\frac{|\cC|}{|G|} \frac{\cD(K/k)^\sigma}{\exp(c_3(\sigma \log \cD(K/k))^{1/2}|G|^{-1/2}n_k^{-1/2})} \le \frac12 \frac{|\cC|}{|G|}\mathrm{Li}(\cD(K/k)^\sigma)
\end{equation*}
and 
\begin{equation*}
\frac12 \frac{|\cC|}{|G|}\mathrm{Li}(\cD(K/k)^\sigma) \gg_{|G|,n,D_k,\sigma} \frac{D_K^\sigma}{\log D_K}.
\end{equation*}
The bound in Proposition \ref{manyideals} is obtained by choosing $D_0=\max \{D_0^{'},D_1^{'} \}$.
\end{proof}

By Proposition \ref{manyideals}, we can choose $M=D_K^{\frac{1}{2\ell(d-1)}-\ep}$ for any $\ep>0$ in Lemma \ref{ev07}, obtaining the bound in Theorem \ref{genltor}.

\end{section}

\section*{Acknowledgement}

The contents in this paper are part of the author's Ph.D. dissertation at Duke University. The author thanks his advisor, Lillian Pierce, who suggested the topic and gave much guidance during the research, and Jesse Thorner for pointing an error in the previous version. The author also thanks Jayce Getz, Robert Lemke Oliver, Jiuya Wang, and Asif Zaman for helpful discussions.

\footnotesize{}

\bibliographystyle{alpha}

\bibliography{NoThBibliography}

  \textsc{Department of Mathematics, Duke University, 120 Science Drive, Durham NC 27708 USA} \par  
  \textit{E-mail address}: \texttt{chen.an@duke.edu, chen.an.nku@gmail.com} \par
  \addvspace{\medskipamount}

\end{document}